\documentclass[lang=american,preliminary]{ems-icm-edited}
\pdfoutput=1

\newcommand{\arXiv}[1]{arXiv:\href{http://arXiv.org/abs/#1}{#1}}
\newcommand{\mydoi}[1]{doi:\href{https://doi.org/#1}{#1}}
\usepackage{tikz}
\usepackage{pgfplots}
\pgfplotsset{compat=1.17}
\usepackage{microtype}

\theoremstyle{plain}
\newtheorem{theorem}{Theorem}

\numberwithin{theorem}{section}

\numberwithin{equation}{section}

\newcommand{\R}{\mathbb{R}}
\newcommand{\Z}{\mathbb{Z}}
\newcommand{\CC}{\mathbb{C}}
\newcommand{\vol}[1]{\mathop{\textup{vol}}\mathopen{}\big(#1\big)\mathclose{}}
\newcommand{\SL}{\mathop{\textup{SL}}\nolimits}
\newcommand{\uhp}{\mathcal{H}}
\renewcommand{\Im}{\operatorname{Im}}
\renewcommand{\Re}{\operatorname{Re}}

\begin{document}

\firstpage{1}
\doinumber{213}
\volume{1}
\copyrightyear{2022}

\title{The work of Maryna\\ Viazovska}
\titlemark{The work of Maryna Viazovska}

\emsauthor{1}{Henry Cohn}{H.~Cohn}

\emsaffil{1}{Microsoft Research New England, One Memorial Drive, Cambridge, MA 02140, USA \email{cohn@microsoft.com}}

\classification[11F03, 11H31]{52C17}


\begin{abstract}
On July 5th, 2022, Maryna Viazovska was awarded a Fields Medal for her solution of the sphere packing problem in eight dimensions, as well as further contributions to related extremal problems and interpolation problems in Fourier analysis. This article explains some of the ideas behind her work to a broad mathematical audience.
\end{abstract}

\maketitle

\section{Introduction}

The sphere packing problem asks how we can fill as large a fraction of space as possible with congruent balls, if they are not allowed to overlap except tangentially.\footnote{To state the problem precisely, ``as large a fraction as possible'' must be made precise. One way to do so is by taking a limit of the packing problem in a bounded region as its size grows relative to the sphere radius. The sphere packing problem turns out to be very robust, in the sense that just about all reasonable formulations are equivalent.} This problem sits at the interface between many branches of mathematics, and of science more generally, with connections ranging from materials science to information theory. Sphere packing is a natural problem in Euclidean geometry, with a simple statement, and one might expect an equally elementary and self-contained solution. Instead, the topic is dominated by unexpected connections.

Before Viazovska's breakthrough work, the optimal sphere packing density was known only in one, two, and three dimensions. One dimension is trivial, because intervals can tile the real line with density~$1$. Two dimensions is not trivial, but Thue \cite{T1892} showed that arranging six neighbors around each disk is optimal, with density~$\pi/\sqrt{12} = 0.9068\dots$. Three dimensions was solved by Hales \cite{H2005} via an ingenious and elaborate computer-assisted proof, which has since been formally verified \cite{Hplus2017}. The unsurprising answer is shown in Figure~\ref{figure:3d}: optimal two-dimensional layers are nestled together as densely as possible, to achieve density $\pi/\sqrt{18} = 0.7404\dots$.

\begin{figure}[t]
\begin{center}
\includegraphics[width=2in]{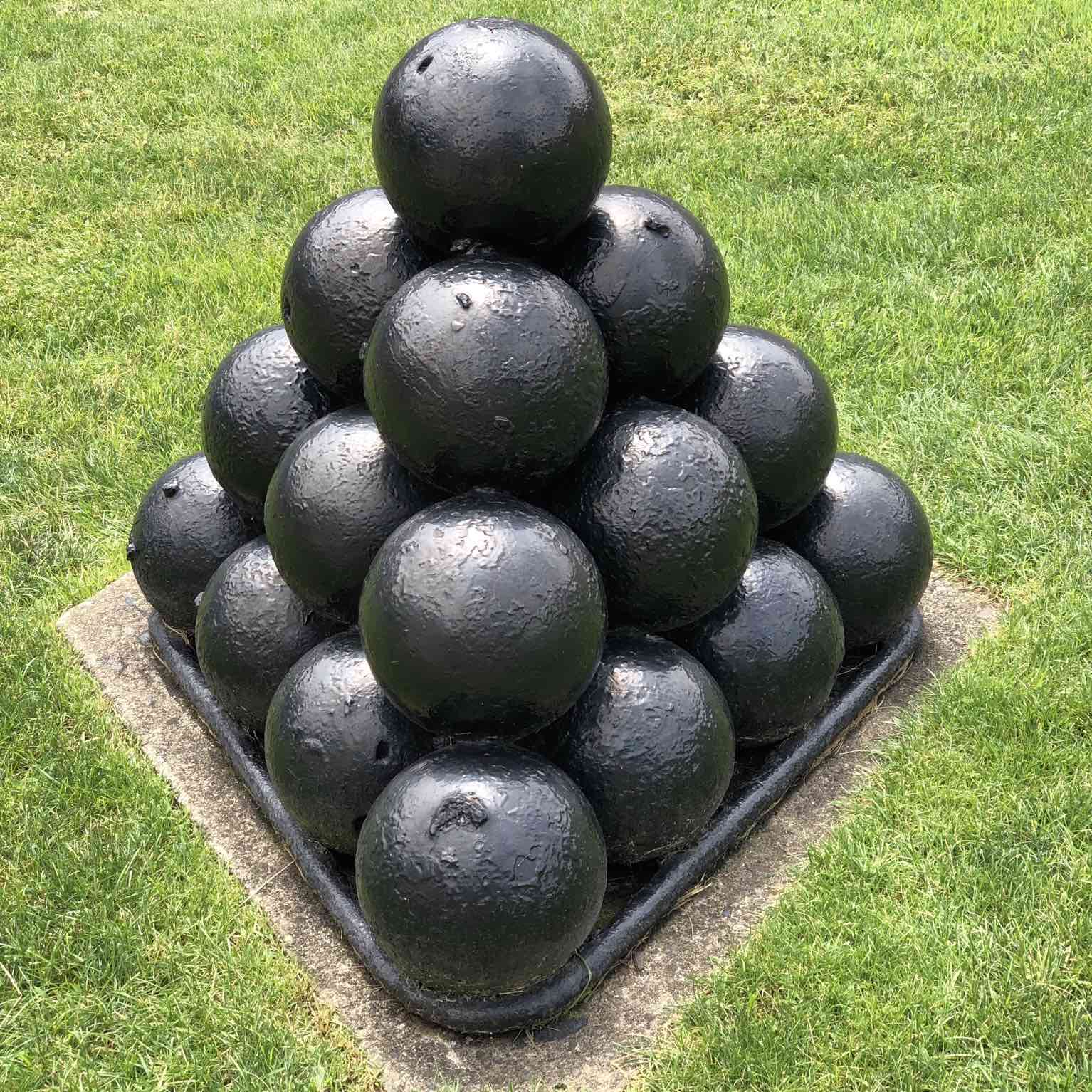}
\end{center}
\caption{\\An optimal packing of cannonballs.}
\label{figure:3d}
\end{figure}

These prior results paint a misleading picture of what happens in higher dimensions. Stacking optimal layers from the previous dimension generally produces suboptimal packings, and nobody has any idea what the densest sphere packings might be in most dimensions. We do not even know whether they should be crystalline or disordered.

High-dimensional packings are not merely of pure mathematical interest, but also important for practical applications, because sphere packings are error-correcting codes for a continuous communication channel (such as radio). In this model, the packing is in an abstract signal space, whose dimension is the number of measurements used to characterize the signal and is generally much larger than three.

There does not seem to be any simple pattern in the optimal packings that persists across many dimensions, and the best upper and lower bounds known for the packing density in $\R^d$ remain exponentially far apart as $d$ grows. However, a handful of dimensions stand out as special, most notably eight and twenty-four dimensions. These dimensions feature exceptional packings, namely the $E_8$ root lattice and the Leech lattice $\Lambda_{24}$, with remarkable symmetries and numerous connections to different branches of mathematics. Thanks to Viazovska's work \cite{V2017, CKMRV2017}, we now know that they are truly optimal. The jump from three dimensions to eight and twenty-four in the known solutions is remarkable, and it illustrates the exceptional nature of these packings.

The $E_8$ and Leech lattices had long been viewed as the most compelling candidates for further solutions of the sphere packing problem. However, a direct geometric proof seems infeasible: it is natural to try to work with a decomposition of space into cells, but the curse of dimensionality means we are faced with an unmanageable number of potential cell shapes and ways they could adjoin each other. Perhaps there exists a proof along these lines, but nobody has found a workable approach.

Instead, Viazovska proved the optimality of $E_8$ via a dramatic new connection to the theory of modular forms, following which she and several collaborators extended her ideas to the case of the Leech lattice:

\begin{theorem}[Viazovska \cite{V2017}] \label{thm:dim8}
The $E_8$ root lattice achieves the optimal sphere packing density in $\R^8$, namely $\pi^4/384$.
\end{theorem}

\begin{theorem}[Cohn, Kumar, Miller, Radchenko, and Viazovska \cite{CKMRV2017}] \label{thm:dim24}
The Leech lattice $\Lambda_{24}$ achieves the optimal sphere packing density in $\R^{24}$, namely $\pi^{12}/12!$.
\end{theorem}

As Peter Sarnak said at the time \cite{K2016}, her paper \cite{V2017} is ``stunningly simple, as all great things are.'' This simplicity is characteristic of Viazovska's work: she has a gift for linking concepts and posing bold conjectures, and these insights lead her to striking arguments. Her proofs engage directly with the heart of the matter, without any extraneous complications. Of course, simple is very much not the same thing as easy. What makes her work extraordinary is how different her ideas are from what came before.

In the remainder of this article, we will examine Viazovska's proof of the optimality of $E_8$, as well as its motivation and place in mathematics more broadly. In particular, this article can serve as an introduction and guide to Viazovska's techniques, alongside other expositions \cite{dLV2016, C2017}. For background on sphere packing and lattices, see \cite{CS1999,E2013,T1983}.

Of course we should keep in mind that this topic represents only one strand of Viazovska's research. For example, \cite{BRV2013} is a beautiful and decisive paper on a quite different topic.
What will she be known for in twenty or thirty years? I look forward to finding out.

\section{The past}

Before we turn to Viazovska's proof, we will need some background. In this section, we will construct the $E_8$ lattice and explain a method for proving upper bounds for the sphere packing density.

Sphere packings can be constructed in many ways, among which lattice packings are the simplest possibility.
A \emph{lattice packing} of spheres centers the spheres at the points of a \emph{lattice} $\Lambda$ in $\R^d$, i.e., a discrete subgroup of $\R^d$ of rank~$d$, or equivalently the integral span of a basis of $\R^d$. There is no reason why an optimal sphere packing should have this algebraic structure, and for example the best sphere packing known in $\R^{10}$ does not. However, many of the best sphere packings known in low dimensions are lattice packings.

To form a packing from a lattice $\Lambda$, we must choose the sphere radius $r$ so that neighboring spheres do not overlap. Specifically, we should take
\[
r = \tfrac{1}{2} \min_{x \in \Lambda \setminus \{0\}} |x|.
\]
The volume of a sphere of radius $r$ in $\R^d$ is $\pi^{d/2} r^n/(d/2)!$, where $(d/2)!$ means $\Gamma(d/2+1)$ when $d$ is odd, and the \emph{density} of the overall packing (i.e., the fraction of space covered by the balls) is the sphere volume times the number of spheres per unit volume in space. Let $\vol{\R^d/\Lambda}$ denote the \emph{covolume} of the lattice, i.e., the volume of the quotient torus, or equivalently the absolute value of the determinant of a lattice basis. Then the number of spheres per unit volume in space is $1/ \vol{\R^d/\Lambda}$, and so the lattice packing density is 
\[
\frac{\pi^{d/2} r^n}{(d/2)!  \vol{\R^d/\Lambda}}.
\]

One of the most remarkable lattices is the \emph{$E_8$ root lattice}, which originated in Lie theory but has since become widespread across mathematics. We will see below how to obtain $E_8$ as a modification of the
$D_d$ lattice, the checkerboard lattice in $d$ dimensions, which is defined by
\[
D_d = \{(x_1,\dots,x_d) \in \Z^d : \text{$x_1+\dots+x_d$ is even}\}.
\]
In other words, $D_d$ simply omits every other point in the cubic lattice $\Z^d$. As a special case, $D_3$ is the face-centered cubic lattice in three dimensions, which Hales showed achieves the optimal sphere packing density \cite{H2005}, and $D_4$ and $D_5$ are the best packings known in their dimensions. However, $D_d$ is not optimal beyond five dimensions.

The problem with $D_d$ in higher dimensions is that its holes are too large. A \emph{hole} is a point in space that is a local maximum for distance from the lattice. There are two types of holes in $D_d$, shallow holes at distance~$1$ from the lattice, such as $(1,0,\dots,0)$, and deep holes at distance $\sqrt{d/4}$ from the lattice, such as $(\tfrac{1}{2},\tfrac{1}{2},\dots,\tfrac{1}{2})$. As $d \to \infty$, so does $\sqrt{d/4}$, and so the deep holes become large enough to fit enormous numbers of additional spheres. In particular, $D_d$ cannot be optimal when $d$ is large.

When $d=8$, something beautiful happens. The distance $\sqrt{8/4}$ from a deep hole to the lattice exactly equals the distance $\sqrt{2}$ between lattice points in $D_8$, and that means the deep holes are just large enough to be filled with additional spheres. If we plug these holes with spheres, then the resulting packing is the union of $D_8$ with its translate $D_8 + (\tfrac{1}{2},\tfrac{1}{2},\dots,\tfrac{1}{2})$. It is not hard to check that this packing is a lattice (it amounts to the fact that ${2 \cdot (\tfrac{1}{2},\tfrac{1}{2},\dots,\tfrac{1}{2})} \in D_8$), which is called the \emph{$E_8$ root lattice}.

The $E_8$ lattice packing has packing radius $r = \sqrt{2}/2$ and covolume $\vol{\R^8/E_8} = \vol{\R^8/D_8}/2 = 1$, and so it has a packing density of $\pi^4/384 = 0.2536\dots$. It is by no means obvious that this construction is optimal. In fact, the construction feels a little ad hoc. However, the $E_8$ lattice turns out to be far more beautiful and symmetric than its construction indicates. For example, see Figure~\ref{figure:E8} for a view of $E_8$ with $30$-fold symmetry.
This is a common pattern with exceptional structures in mathematics: they are typically obtained by piecing together several substructures that each have less symmetry individually.

\begin{figure}[t]
\begin{center}
\includegraphics[width=3in,height=3in]{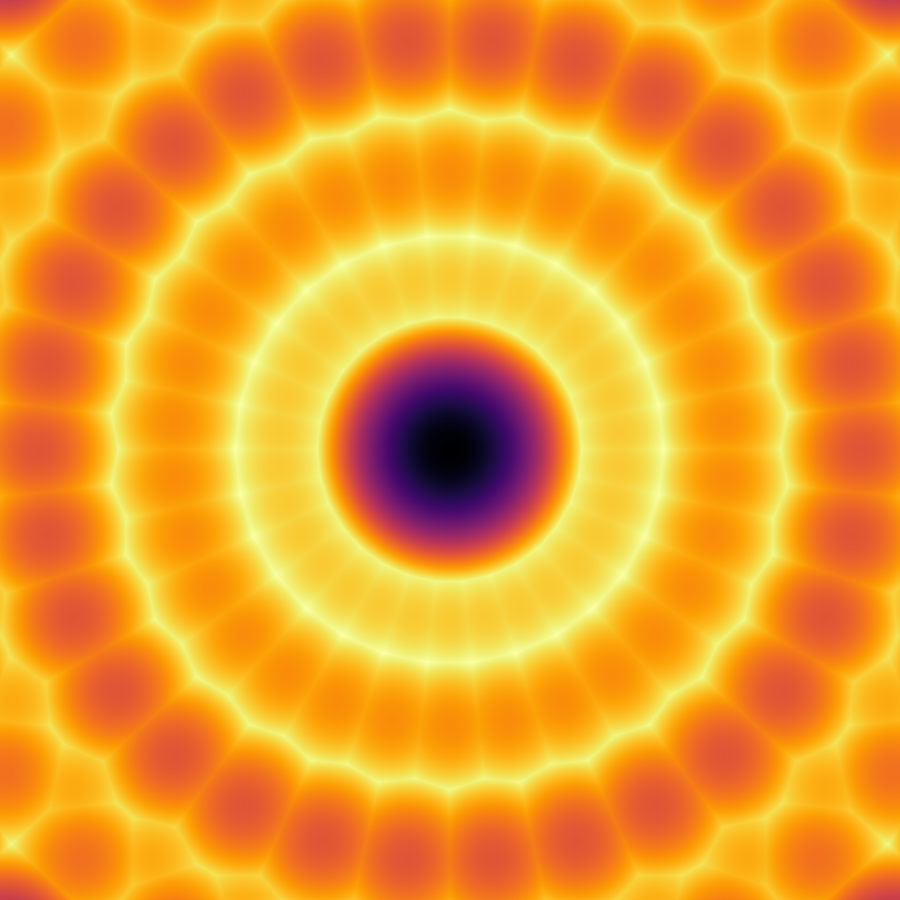}
\end{center}
\caption{\\A two-dimensional cross section of $\R^8$ through a Coxeter plane of $E_8$, colored according to the squared distance to the nearest point in $E_8$ (dark is close) and inspired by \cite{M2017}.}
\label{figure:E8}
\end{figure}

Now that we have the $E_8$ lattice, the next question is how we could try to obtain a matching upper bound for the sphere packing density in eight dimensions. Obtaining a matching bound seems completely infeasible in most dimensions, but in a few special dimensions bounds based on harmonic analysis work remarkably well. This idea, called the \emph{linear programming bound}, goes back to a fundamental paper by Delsarte \cite{D1972} on error-correcting codes, and the corresponding bound for sphere packings was developed by Cohn and Elkies \cite{CE2003}.

The linear programming bound is formulated in terms of the \emph{Fourier transform} $\widehat{f}$ of an integrable function $f \colon \R^d \to \CC$, which we will normalize as
\[
\widehat{f}(y) = \int_{\R^d} f(x) e^{-2\pi i \langle x,y \rangle} \, dx,
\]
where $\langle \cdot,\cdot \rangle$ is the usual inner product on $\R^d$. Recall that the Fourier transform decomposes $f$ into complex exponentials; in signal processing terms, it amounts to identifying the frequencies that occur in a signal and their relative magnitudes. This decomposition amounts to the \emph{Fourier inversion theorem}: if $\widehat{f}$ is integrable as well, then
\[
f(x) = \int_{\R^d} \widehat{f}(y) e^{2\pi i \langle x,y \rangle} \, dy. 
\]
In other words, the Fourier transform is very nearly its own inverse, with a single sign change being the only difference. Note that $\widehat{f}$ is generally complex-valued, even if $f$ is real-valued, but $\widehat{f}$ is real-valued if $f$ is real-valued and an even function.

We will also need a few types of well-behaved functions. A function $f \colon \R^d \to \R$ is called \emph{rapidly decreasing} if $f(x) = O\big(|x|^{-c}\big)$ as $|x| \to \infty$ for every constant $c>0$, and a \emph{Schwartz function} is a smooth function such that it and all its iterated partial derivatives (of every order) are rapidly decreasing. Schwartz functions are arguably the best-behaved functions in harmonic analysis. Much of what we will discuss can be generalized somewhat beyond Schwartz functions, but they are all Viazovska needed to solve the sphere packing problem.

We can now state the linear programming bound for sphere packing:

\begin{theorem}[Cohn and Elkies \cite{CE2003}] 
\label{thm:LPbound}
Let $f \colon \R^d \to \R$ be an even Schwartz function and $r$ a positive real number. If 
\begin{enumerate}
\item $f(x) \le 0$ for all $x \in \R^d$ satisfying $|x| \ge r$, 
\item $\widehat{f}(y) \ge 0$ for all $y \in \R^d$, and
\item $f(0)=\widehat{f}(0)=1$,
\end{enumerate}
then the optimal sphere packing density in $\R^d$ is at most $\vol{B_{r/2}^d} = \pi^{d/2} (r/2)^d/(d/2)!$.
\end{theorem}

\begin{figure}[t]
\begin{center}
\begin{tikzpicture}
\begin{axis}[every x tick/.style={black}, every y tick/.style={black},
scale only axis,
ymode = log,
xmin=0, xmax=32, ymin=0.00001, ymax=1,
xlabel=\textup{Dimension $d$},
ylabel=\textup{Sphere packing density in $\mathbb{R}^d$},
domain = 0:36,
enlargelimits = false,
xtick = {0,4,...,36},
extra x ticks = {1,...,36},
extra x tick labels={},
yticklabels={$10^{-6}$, $10^{-5}$, $10^{-4}$, $10^{-3}$, $10^{-2}$, $10^{-1}$, $1$},
legend style={at={(0.6,0.3)},anchor=north east},
width=4in,
height=3in
]
\addplot +[mark=none, densely dashed, color=black] plot coordinates {
(1,1.000000000000000)
(2,0.9068996821171089)
(3, 0.779746762 )
(4, 0.647704966 )
(5, 0.524980022 )
(6, 0.417673416 )
(7, 0.327455611  )
(8, 0.253669508 )
(9, 0.194555339  )
(10, 0.147953479 )
(11, 0.111690766  )
(12, 0.083775831 )
(13,0.06248170014568450)
(14,0.04636448923374530)
(15,0.03424826203368300)
(16,0.02519413072133100)
(17,0.01846409033506490)
(18,0.01348534044508620)
(19,0.009817955139543800)
(20,0.007127053603376300)
(21,0.005159660394817600)
(22,0.003725941968920600)
(23,0.002684279886429100)
(24,0.001929574309403923)
(25,0.001384190722285700)
(26,0.0009910238892216000)
(27,0.0007082297958617000)
(28,0.0005052542161057000)
(29,0.0003598581852089000)
(30,0.0002559028743732000)
(31,0.0001817083813917000)
(32,0.0001288432887595000)
};
\addplot +[mark=none, mark options={fill=black,draw=black},color=black] plot
coordinates {
(1,1.000000000000000)
(2,0.9068996821171089)
(3,0.7404804896930610)
(4,0.6168502750680849)
(5,0.4652576133092586)
(6,0.3729475455820649)
(7,0.2952978731457126)
(8,0.2536695079010480)
(9,0.1457748758081711)
(10,0.09961578280770881)
(11,0.06623802700980118)
(12,0.04945417662424406)
(13,0.03201429216034980)
(14,0.02162409608244711)
(15,0.01685757065676270)
(16,0.01470816439743083)
(17,0.008811319182321190)
(18,0.006167898125331257)
(19,0.004120806279768668)
(20,0.003394581410712645)
(21,0.002465884711502463)
(22,0.002451034044121183)
(23,0.001905328193426062)
(24,0.001929574309403923)
(25,0.0006772120097731805)
(26,0.0002692200504338089)
(27,0.0001575943907278669)
(28,0.0001046381049248457)
(29,0.00003414464690742249)
(30,0.00002191535344783022)
(31,0.00001183776518593385)
(32,0.00001104074930885985)
};      
\legend{\quad Linear programming bound, \quad Sphere packing density}
\end{axis}
\end{tikzpicture}
\end{center}
\caption{\\A plot of the numerically computed linear programming bound \cite{ACHLT2020} and the best sphere packing density currently known \cite{CS1999}.}
\label{figure:plot}
\end{figure}

This theorem produces an upper bound for the packing density from a function $f$ satisfying certain inequalities, but it says nothing about how to choose $f$ to optimize the bound. Numerical optimization can produce good choices for $f$, which yield the bounds shown in Figure~\ref{figure:plot}. These bounds are rigorous, but it is possible that other functions may produce even better bounds.

As one can see in Figure~\ref{figure:plot}, the bounds in eight and twenty-four dimensions appear sharp. Numerical optimization will not yield an exactly sharp bound, but it seems to come as close as desired.
Based on data of this sort as well as analogies with other problems in coding theory, Cohn and Elkies conjectured the existence of \emph{magic functions} $f$ that would solve the sphere packing problem exactly in $\R^8$ and $\R^{24}$, by achieving $r=\sqrt{2}$ and $r=2$, respectively. Note that this is not because the bound dips lower in these dimensions, but rather because the optimal packings rise up to meet it. No other dimensions greater than~$2$ seem to have a sharp linear programming bound, and it seems unlikely that others exist, but no proof is known, and the bound has been exactly optimized only for $d=1$, $8$, and $24$.

The heart of Viazovska's breakthrough lies in the construction of the magic functions.
What should $f$ look like if we are to obtain a sharp bound? There are some simple criteria, which we can obtain from the proof of Theorem~\ref{thm:LPbound}. In this article we will examine a proof for just the special case of lattices, but the theorem can be proved in full generality by combining the same technique with a little additional algebra. The argument is based on the \emph{Poisson summation formula}, which says that if $f \colon \R^d \to \CC$ is a Schwartz function, $\Lambda$ is a lattice in $\R^d$, and $\Lambda^*$ is its \emph{dual lattice} (i.e., the lattice generated by the dual basis of any basis of $\Lambda$ with respect to the inner product $\langle \cdot, \cdot \rangle$), then
\[
\sum_{x \in \Lambda} f(x) = \frac{1}{\vol{\R^d/\Lambda}} \sum_{y \in \Lambda^*} \widehat{f}(y).
\]

\begin{proof}[Proof of Theorem~\ref{thm:LPbound} for lattice packings]
The sphere packing problem is scaling-invariant, and so we can use spheres of radius $r/2$. Let $\Lambda$ be any lattice packing with packing radius $r/2$, which means $|x| \ge r$ for $x \in \Lambda \setminus \{0\}$.
If $f$ satisfies the hypotheses of Theorem~\ref{thm:LPbound}, then $f(x) \le 0$ for $x \in \Lambda \setminus \{0\}$ and $\widehat{f}(y) \ge 0$ for all $y$, from which it follows that 
\[
1 = f(0) \ge \sum_{x \in \Lambda} f(x) = \frac{1}{\vol{\R^d/\Lambda}} \sum_{y \in \Lambda^*} \widehat{f}(y) \ge \frac{\widehat{f}(0)}{ \vol{\R^d/\Lambda}} = \frac{1}{ \vol{\R^d/\Lambda}}.
\]
Therefore the packing density $\vol{B_{r/2}^d}/\vol{\R^d/\Lambda}$ is bounded above by $\vol{B_{r/2}^d}$, as desired.
\end{proof}

A first observation is that we can assume without loss of generality that $f$ is radial, i.e., $f(x)$ depends only on $|x|$. 
This reason is that we can replace $f$ with the average of its rotations about the origin, because all the constraints are linear and rotation-invariant. One might wonder whether non-radial functions could be helpful conceptually even if they are not needed, but so far the answer appears to be no. Instead, Viazovska's work turns out to lead to a wonderful new theory of interpolation for radial functions. We will henceforth assume $f$ is radial, and when $t \in [0,\infty)$ we will write $f(t)$ for the common value $f(x)$ with $|x|=t$, as well as $f'(t)$ for the radial derivative.

Now if we examine the central inequality in the proof of Theorem~\ref{thm:LPbound} for lattices, we can see when it could be sharp. To obtain a sharp bound, all of the discarded terms in the inequality must vanish: we must have $f(x) = 0$ for $x \in \Lambda\setminus\{0\}$ and $\widehat{f}(y)=0$ for $y \in \Lambda^*\setminus\{0\}$. In other words, $f$ must vanish on the nonzero distances between lattice points, and $\widehat{f}$ must vanish on the nonzero distances between dual lattice points.

One can check directly from the construction of $E_8$ given above that $E_8^*=E_8$ and that the vector lengths in $E_8$ are all square roots of even integers. Furthermore, it turns out that each distance $\sqrt{2n}$ with $n \ge 0$ actually occurs in $E_8$. We should therefore have $r = \sqrt{2}$ in Theorem~\ref{thm:LPbound}, and the magic function $f$ should have a sign change at radius $\sqrt{2}$, followed by double roots at $\sqrt{2n}$ for $n \ge 2$, as indicated in Figure~\ref{figure:diagram}. In other words, we  wish to control the behavior of $f$ and $\widehat{f}$ to second order at these points, i.e., control both the values $f\big(\sqrt{2n}\big)$ and $\widehat{f}\big(\sqrt{2n}\big)$ and the radial derivatives $f'\big(\sqrt{2n}\big)$ and $\widehat{f}\,'\big(\sqrt{2n}\big)$.

\begin{figure}[t]
\begin{center}
\begin{tikzpicture}[scale=1.2]
\draw[black!50] (0,-0.75) -- (0,1.5);
\draw[black!50] (-0.1/1.2,1.5) -- (0.1/1.2,1.5);
\draw[black!50] (0,0) -- (4.5,0);
\draw (2.25,1) node {$f$};
\draw[black!50] (1,-0.1/1.2) -- (1,0.1/1.2); \draw (0.75,-0.1/1.2) node[below] {$\sqrt{2}$};
\draw[black!50] (2,-0.1/1.2) -- (2,0.1/1.2); \draw (2,-0.1/1.2) node[below] {$\sqrt{4}$};
\draw[black!50] (3,-0.1/1.2) -- (3,0.1/1.2); \draw (3,-0.1/1.2) node[below] {$\sqrt{6}$};
\draw[black!50] (4,-0.1/1.2) -- (4,0.1/1.2); \draw (4,-0.1/1.2) node[below] {$\sqrt{8}$};
\draw (0,1.5) to[out=0,in=106] (1,0) to[out=286,in=180] (1.4,-0.5)
to[out=0,in=180] (2,0) to[out=0,in=180] (2.5,-0.25)
to[out=0,in=180] (3,0) to[out=0,in=180] (3.5,-0.125)
to[out=0,in=180] (4,0) to[out=0,in=180] (4.5,-0.0625);
\end{tikzpicture}
\hskip 1cm
\begin{tikzpicture}[scale=1.2]
\draw[black!50] (0,-0.75) -- (0,1.5);
\draw[black!50] (-0.1/1.2,1.5) -- (0.1/1.2,1.5);
\draw[black!50] (0,0) -- (4.5,0);
\draw (2.25,1) node {$\widehat{f}$};
\draw[black!50] (1,-0.1/1.2) -- (1,0.1/1.2); \draw (0.75,-0.1/1.2) node[below] {$\sqrt{2}$};
\draw[black!50] (2,-0.1/1.2) -- (2,0.1/1.2); \draw (2,-0.1/1.2) node[below] {$\sqrt{4}$};
\draw[black!50] (3,-0.1/1.2) -- (3,0.1/1.2); \draw (3,-0.1/1.2) node[below] {$\sqrt{6}$};
\draw[black!50] (4,-0.1/1.2) -- (4,0.1/1.2); \draw (4,-0.1/1.2) node[below] {$\sqrt{8}$};
\draw (0,1.5) to[out=0,in=180] (1,0) to[out=0,in=180] (1.5,0.4)
to[out=0,in=180] (2,0) to[out=0,in=180] (2.5,0.2)
to[out=0,in=180] (3,0) to[out=0,in=180] (3.5,0.1)
to[out=0,in=180] (4,0) to[out=0,in=180] (4.5,0.05);
\end{tikzpicture}
\end{center}
\caption{\\This schematic diagram, which is taken from \cite{C2017}, shows the roots of the magic
function $f$ and its Fourier transform $\widehat{f}$ in eight dimensions.
It is not a plot of the actual function, which decreases very rapidly. See Figure~\ref{figure:magicf}
for an actual plot.}
\label{figure:diagram}
\end{figure}

\begin{figure}[t]
\begin{center}
\begin{tikzpicture}[scale=0.9]
\draw[black!50] (0,0) -- (6,0);
\draw[black!50] (0,-1.27) -- (0,2);
\draw[black!50] (-0.1/0.9,2) -- (0.1/0.9,2);
\draw[black!50] (2,-0.1/0.9) -- (2,0.1/0.9);
\draw[black!50] (4,-0.1/0.9) -- (4,0.1/0.9);
\draw[black!50] (6,-0.1/0.9) -- (6,0.1/0.9);
\draw (3,-1.5) node {$x \mapsto f(x)$};
\draw (0,2)--(3/100,1.9987854)--(3/50,1.9951468)--(9/100,1.9890995)--(3/25,1.9806688)--(3/20,1.9698899)--(9/50,1.9568075)--(21/100,1.9414755)--(6/25,1.9239564)--(27/100,1.9043209)--(3/10,1.8826473)--(33/100,1.8590207)--(9/25,1.8335326)--(39/100,1.8062799)--(21/50,1.7773645)--(9/20,1.7468922)--(12/25,1.7149721)--(51/100,1.6817159)--(27/50,1.6472371)--(57/100,1.6116502)--(3/5,1.5750702)--(63/100,1.5376115)--(33/50,1.4993879)--(69/100,1.4605116)--(18/25,1.4210927)--(3/4,1.3812388)--(39/50,1.3410544)--(81/100,1.3006407)--(21/25,1.2600952)--(87/100,1.2195115)--(9/10,1.1789786)--(93/100,1.1385815)--(24/25,1.0984003)--(99/100,1.0585106)--(51/50,1.0189833)--(21/20,0.97988437)--(27/25,0.94127527)--(111/100,0.90321264)--(57/50,0.86574856)--(117/100,0.82893060)--(6/5,0.79280192)--(123/100,0.75740145)--(63/50,0.72276398)--(129/100,0.68892040)--(33/25,0.65589777)--(27/20,0.62371958)--(69/50,0.59240586)--(141/100,0.56197341)--(36/25,0.53243593)--(147/100,0.50380422)--(3/2,0.47608634)--(153/100,0.44928777)--(39/25,0.42341159)--(159/100,0.39845858)--(81/50,0.37442741)--(33/20,0.35131478)--(42/25,0.32911550)--(171/100,0.30782266)--(87/50,0.28742772)--(177/100,0.26792062)--(9/5,0.24928990)--(183/100,0.23152274)--(93/50,0.21460511)--(189/100,0.19852184)--(48/25,0.18325668)--(39/20,0.16879235)--(99/50,0.15511070)--(201/100,0.14219270)--(51/25,0.13001851)--(207/100,0.11856761)--(21/10,0.10781880)--(213/100,0.097750275)--(54/25,0.088339728)--(219/100,0.079564370)--(111/50,0.071401015)--(9/4,0.063826138)--(57/25,0.056815946)--(231/100,0.050346437)--(117/50,0.044393476)--(237/100,0.038932852)--(12/5,0.033940355)--(243/100,0.029391835)--(123/50,0.025263281)--(249/100,0.021530878)--(63/25,0.018171081)--(51/20,0.015160679)--(129/50,0.012476859)--(261/100,0.010097268)--(66/25,0.0080000781)--(267/100,0.0061640369)--(27/10,0.0045685276)--(273/100,0.0031936175)--(69/25,0.0020201048)--(279/100,0.0010295609)--(141/50,0.00020436815)--(57/20,-0.00047224819)--(72/25,-0.0010161928)--(291/100,-0.0014424790)--(147/50,-0.0017652128)--(297/100,-0.0019975833)--(3,-0.0021518586)--(303/100,-0.0022393872)--(153/50,-0.0022706060)--(309/100,-0.0022550526)--(78/25,-0.0022013841)--(63/20,-0.0021174001)--(159/50,-0.0020100706)--(321/100,-0.0018855684)--(81/25,-0.0017493049)--(327/100,-0.0016059697)--(33/10,-0.0014595730)--(333/100,-0.0013134900)--(84/25,-0.0011705076)--(339/100,-0.0010328718)--(171/50,-0.00090233642)--(69/20,-0.00078021143)--(87/25,-0.00066741138)--(351/100,-0.00056450263)--(177/50,-0.00047174953)--(357/100,-0.00038915867)--(18/5,-0.00031652117)--(363/100,-0.00025345244)--(183/50,-0.00019942922)--(369/100,-0.00015382367)--(93/25,-0.00011593435)--(15/4,-8.5013941 E-5)--(189/50,-6.0293648 E-5)--(381/100,-4.1004344 E-5)--(96/25,-2.6394416 E-5)--(387/100,-1.5744464 E-5)--(39/10,-8.3789703 E-6)--(393/100,-3.6751161 E-6)--(99/25,-1.0689616 E-6)--(399/100,-5.9215183 E-8)--(201/50,-2.0885328 E-7)--(81/20,-1.1448593 E-6)--(102/25,-2.5563584 E-6)--(411/100,-4.1914266 E-6)--(207/50,-5.8528487 E-6)--(417/100,-7.3930893 E-6)--(21/5,-8.7087261 E-6)--(423/100,-9.7345770 E-6)--(213/50,-1.0437728 E-5)--(429/100,-1.0811650 E-5)--(108/25,-1.0870551 E-5)--(87/20,-1.0644111 E-5)--(219/50,-1.0172675 E-5)--(441/100,-9.5030083 E-6)--(111/25,-8.6846283 E-6)--(447/100,-7.7667646 E-6)--(9/2,-6.7959269 E-6)--(453/100,-5.8140710 E-6)--(114/25,-4.8573235 E-6)--(459/100,-3.9552170 E-6)--(231/50,-3.1303738 E-6)--(93/20,-2.3985742 E-6)--(117/25,-1.7691368 E-6)--(471/100,-1.2455411 E-6)--(237/50,-8.2622160 E-7)--(477/100,-5.0547012 E-7)--(24/5,-2.7438476 E-7)--(483/100,-1.2181525 E-7)--(243/50,-3.5260181 E-8)--(489/100,-1.6818469 E-9)--(123/25,-8.2130521 E-9)--(99/20,-4.2739253 E-8)--(249/50,-9.4347745 E-8)--(501/100,-1.5364309 E-7)--(126/25,-2.1293451 E-7)--(507/100,-2.6630620 E-7)--(51/10,-3.0958574 E-7)--(513/100,-3.4022826 E-7)--(129/25,-3.5713584 E-7)--(519/100,-3.6043183 E-7)--(261/50,-3.5120910 E-7)--(21/4,-3.3126981 E-7)--(132/25,-3.0287206 E-7)--(531/100,-2.6849599 E-7)--(267/50,-2.3063925 E-7)--(537/100,-1.9164822 E-7)--(27/5,-1.5358895 E-7)--(543/100,-1.1815843 E-7)--(273/50,-8.6634928 E-8)--(549/100,-5.9863775 E-8)--(138/25,-3.8273499 E-8)--(111/20,-2.1916413 E-8)--(279/50,-1.0526968 E-8)--(561/100,-3.5912999 E-9)--(141/25,-4.2165204 E-10)--(567/100,-2.3005869 E-10)--(57/10,-2.1965768 E-9)--(573/100,-5.5284309 E-9)--(144/25,-9.5075829 E-9)--(579/100,-1.3525378 E-8)--(291/50,-1.7103983 E-8)--(117/20,-1.9905249 E-8)--(147/25,-2.1728407 E-8)--(591/100,-2.2498528 E-8)--(297/50,-2.2248035 E-8)--(597/100,-2.1093700 E-8)--(6,-1.9211474 E-8);
\end{tikzpicture}
\hskip 1cm
\begin{tikzpicture}[scale=0.9]
\draw[black!50] (0,0) -- (6,0);
\draw[black!50] (0,-1.27) -- (0,2);
\draw[black!50] (-0.1/0.9,2) -- (0.1/0.9,2);
\draw[black!50] (2,-0.1/0.9) -- (2,0.1/0.9);
\draw[black!50] (4,-0.1/0.9) -- (4,0.1/0.9);
\draw[black!50] (6,-0.1/0.9) -- (6,0.1/0.9);
\draw (3,-1.5) node {$x \mapsto e^{2\pi|x|}|x|^{7/2}f(x)/300$};
\draw (0,0)--(3/100,3.0261860 E-9)--(3/50,3.7552649 E-8)--(9/100,1.7004855 E-7)--(3/25,5.0926654 E-7)--(3/20,1.2153226 E-6)--(9/50,2.5110964 E-6)--(21/100,4.6956125 E-6)--(6/25,8.1594083 E-6)--(27/100,1.3401972 E-5)--(3/10,2.1051353 E-5)--(33/100,3.1886068 E-5)--(9/25,4.6859389 E-5)--(39/100,6.7126146 E-5)--(21/50,9.4072101 E-5)--(9/20,0.00012934601)--(12/25,0.00017489436)--(51/100,0.00023299895)--(27/50,0.00030631712)--(57/100,0.00039792486)--(3/5,0.00051136247)--(63/100,0.00065068291)--(33/50,0.00082050255)--(69/100,0.0010260542)--(18/25,0.0012732418)--(3/4,0.0015686975)--(39/50,0.0019198390)--(81/100,0.0023349281)--(21/25,0.0028231293)--(87/100,0.0033945676)--(9/10,0.0040603847)--(93/100,0.0048327931)--(24/25,0.0057251255)--(99/100,0.0067518805)--(51/50,0.0079287605)--(21/20,0.0092727022)--(27/25,0.010801897)--(111/100,0.012535797)--(57/50,0.014495112)--(117/100,0.016701781)--(6/5,0.019178933)--(123/100,0.021950817)--(63/50,0.025042713)--(129/100,0.028480807)--(33/25,0.032292039)--(27/20,0.036503912)--(69/50,0.041144262)--(141/100,0.046240985)--(36/25,0.051821715)--(147/100,0.057913455)--(3/2,0.064542148)--(153/100,0.071732206)--(39/25,0.079505963)--(159/100,0.087883081)--(81/50,0.096879891)--(33/20,0.10650867)--(42/25,0.11677688)--(171/100,0.12768630)--(87/50,0.13923218)--(177/100,0.15140230)--(9/5,0.16417599)--(183/100,0.17752313)--(93/50,0.19140318)--(189/100,0.20576414)--(48/25,0.22054157)--(39/20,0.23565766)--(99/50,0.25102032)--(201/100,0.26652241)--(51/25,0.28204104)--(207/100,0.29743707)--(21/10,0.31255476)--(213/100,0.32722162)--(54/25,0.34124860)--(219/100,0.35443044)--(111/50,0.36654649)--(9/4,0.37736180)--(57/25,0.38662867)--(231/100,0.39408855)--(117/50,0.39947452)--(237/100,0.40251415)--(12/5,0.40293290)--(243/100,0.40045799)--(123/50,0.39482280)--(249/100,0.38577168)--(63/25,0.37306528)--(51/20,0.35648623)--(129/50,0.33584508)--(261/100,0.31098663)--(66/25,0.28179631)--(267/100,0.24820657)--(27/10,0.21020321)--(273/100,0.16783146)--(69/25,0.12120155)--(279/100,0.070493738)--(141/50,0.015962473)--(57/20,-0.042060427)--(72/25,-0.10316384)--(291/100,-0.16685680)--(147/50,-0.23256961)--(297/100,-0.29965671)--(3,-0.36740151)--(303/100,-0.43502329)--(153/50,-0.50168622)--(309/100,-0.56651061)--(78/25,-0.62858627)--(63/20,-0.68698799)--(159/50,-0.74079289)--(321/100,-0.78909948)--(81/25,-0.83104815)--(327/100,-0.86584254)--(33/10,-0.89277155)--(333/100,-0.91123129)--(84/25,-0.92074639)--(339/100,-0.92099012)--(171/50,-0.91180243)--(69/20,-0.89320537)--(87/25,-0.86541502)--(351/100,-0.82884943)--(177/50,-0.78413158)--(357/100,-0.73208719)--(18/5,-0.67373662)--(363/100,-0.61028064)--(183/50,-0.54307999)--(369/100,-0.47362864)--(93/25,-0.40352108)--(15/4,-0.33441417)--(189/50,-0.26798423)--(381/100,-0.20588038)--(96/25,-0.14967536)--(387/100,-0.10081532)--(39/10,-0.060570103)--(393/100,-0.029985950)--(99/25,-0.0098423630)--(399/100,-0.00061513951)--(201/50,-0.0024473675)--(81/20,-0.015130131)--(102/25,-0.038094441)--(411/100,-0.070415589)--(207/50,-0.11083077)--(417/100,-0.15777032)--(21/5,-0.20940237)--(423/100,-0.26369027)--(213/50,-0.31846126)--(429/100,-0.37148460)--(108/25,-0.42055645)--(87/20,-0.46358853)--(219/50,-0.49869701)--(441/100,-0.52428768)--(111/25,-0.53913360)--(447/100,-0.54244104)--(9/2,-0.53389991)--(453/100,-0.51371551)--(114/25,-0.48261844)--(459/100,-0.44185110)--(231/50,-0.39312949)--(93/20,-0.33858077)--(117/25,-0.28065773)--(471/100,-0.22203305)--(237/50,-0.16547730)--(477/100,-0.11372574)--(24/5,-0.069340119)--(483/100,-0.034572292)--(243/50,-0.011237102)--(489/100,-0.00060178533)--(123/25,-0.0032990421)--(99/20,-0.019270006)--(249/50,-0.047742178)--(501/100,-0.087245768)--(126/25,-0.13566990)--(507/100,-0.19035795)--(51/10,-0.24823884)--(513/100,-0.30598886)--(129/25,-0.36021611)--(519/100,-0.40765796)--(261/50,-0.44538014)--(21/4,-0.47096544)--(132/25,-0.48267949)--(531/100,-0.47960182)--(267/50,-0.46171193)--(537/100,-0.42992191)--(27/5,-0.38605057)--(543/100,-0.33273728)--(273/50,-0.27329767)--(549/100,-0.21152770)--(138/25,-0.15146636)--(111/20,-0.097130907)--(279/50,-0.052241642)--(561/100,-0.019954773)--(141/25,-0.0026229346)--(567/100,-0.0016020167)--(57/10,-0.017120931)--(573/100,-0.048227342)--(144/25,-0.092817608)--(579/100,-0.14775332)--(291/50,-0.20906043)--(117/20,-0.27220028)--(147/25,-0.33239559)--(591/100,-0.38498897)--(297/50,-0.42580771)--(597/100,-0.45150599)--(6,-0.45985610);
\end{tikzpicture}
\end{center}
\caption{\\Two plots of Viazovska's magic function in eight dimensions. The first plot is scaled correctly, but it decreases so rapidly that the roots become invisible. The second plot introduces a rescaling to make them visible, based on the asymptotic decay rate.}
\label{figure:magicf}
\end{figure}

How can one construct such a function $f$? The reason this task is difficult is that it involves controlling both $f$ and $\widehat{f}$ simultaneously. Either one is of course easy on its own, but handling both at once introduces profound difficulties. The underlying issue here is Heisenberg's uncertainty principle: in loose terms, whenever you try to pin down $f$, you lose control over $\widehat{f}$, and vice versa. More precisely, we run into Bourgain, Clozel, and Kahane's uncertainty principle for controlling the signs of functions \cite{BCK2010, CG2019}. These seemingly simple inequalities on $f$ and $\widehat{f}$ therefore turn out to be far more subtle than they initially appear.

When Elkies and I proposed this method in 1999, Viazovska was still in secondary school. Without realizing how profoundly difficult the remaining step was, 
I imagined that we had almost solved the sphere packing problem in eight and twenty-four dimensions, and our inability to find the magic functions was extremely frustrating. At first, I worried that someone else would find an easy solution and leave me feeling foolish for not doing it myself. Over time I became convinced that obtaining these functions was in fact difficult, and others also reached the same conclusion. For example, Thomas Hales has said that
``I felt that it would take a Ramanujan to find it'' \cite{K2016}.
Eventually, instead of worrying that someone else would solve it, I began to fear that nobody would solve it, and that I would someday die without knowing the outcome. I am grateful that Viazovska found such a satisfying and beautiful solution, and that she introduced wonderful new ideas for the mathematical community to explore.

\section{Modular forms}

Viazovska's magic function is constructed using modular forms, certain special functions that play an important role in number theory. The theory of modular forms has a reputation for being somewhat forbidding, but the basics are not so difficult, and that is all that is needed for Viazovska's proof. We will outline the needed theory here. For a down to earth introduction to the case of $\SL_2(\Z)$, see Chapter~VII in \cite{S1973}, and for more detailed and general treatments, see \cite{DS2005,CS2017,Z2008}.

We begin with an example of a modular form, namely Eisenstein series. Recall that the Riemann zeta function is defined by
\[
\zeta(s) = \sum_{n=1}^\infty \frac{1}{n^s}
\]
when this sum converges, i.e., when $\Re(s) > 1$. Here we are summing inverse powers of the arithmetic progression $1,2,\dots$, and Euler obtained an exact formula when $s$ is an even integer.
What if we instead wanted to sum inverse powers of a lattice in the complex plane? Setting aside the question of why we would want to do this (the result has deeper significance than one might guess), we could write the result as the \emph{Eisenstein series}
\begin{equation} \label{eq:eisenstein}
E_k(z) = \frac{1}{2\zeta(k)} 
\sum_{(m,n) \in \Z^2 \setminus \{(0,0)\}} \frac{1}{(mz+n)^k}
\end{equation}
for $\Im z  > 0$,
where we are summing over the lattice $\{mz+n : m,n \in \Z\}$, with the exception of the point $(0,0)$ at which the summand blows up. Up to scaling by a complex factor, all two-dimensional lattices are of this form. 

The factor of $1/(2\zeta(k))$ in the definition is merely a convenient normalizing factor, which plays no essential role in the study of $E_k$. 
Unfortunately, the notation $E_k$ conflicts with our name for the $E_8$ root lattice, but that will not cause any ambiguity in practice.

We will restrict our attention to positive integers $k$, so that $(mz+n)^k$ is single-valued. The series \eqref{eq:eisenstein} converges absolutely when $k \ge 3$, but just conditionally when $k=2$. For odd $k$, the $(m,n)$ and $(-m,-n)$ terms cancel and we obtain $E_k(z)=0$, and so only the even cases are interesting.\footnote{This parity phenomenon is essentially the same as in Euler's formula for the zeta function at even integers, which can be viewed as computing $\sum_{n \in \Z\setminus\{0\}} n^{-k}$ explicitly for all integers $k>1$.} Thus, we will focus on $E_k$ for $k$ even and at least $4$.

What does an Eisenstein series look like? Figure~\ref{figure:E4} is a plot of $E_4$, in which
black is zero, white is infinity, and color indicates complex phase \cite{L-D2021}, with the sharp transitions in color occurring at positive real values.
The fractal structure visible in this plot can be explained using two functional equations: 
\[
E_k(z+1) = E_k(z) \quad\text{and}\quad E_k(-1/z) = z^k E_k(z).
\]
These symmetries follow from rearranging the defining series \eqref{eq:eisenstein} when $k>2$, and they are the central equations in the theory of modular forms.

\begin{figure}[t]
\begin{center}
\includegraphics[width=4.07in,height=2.07in]{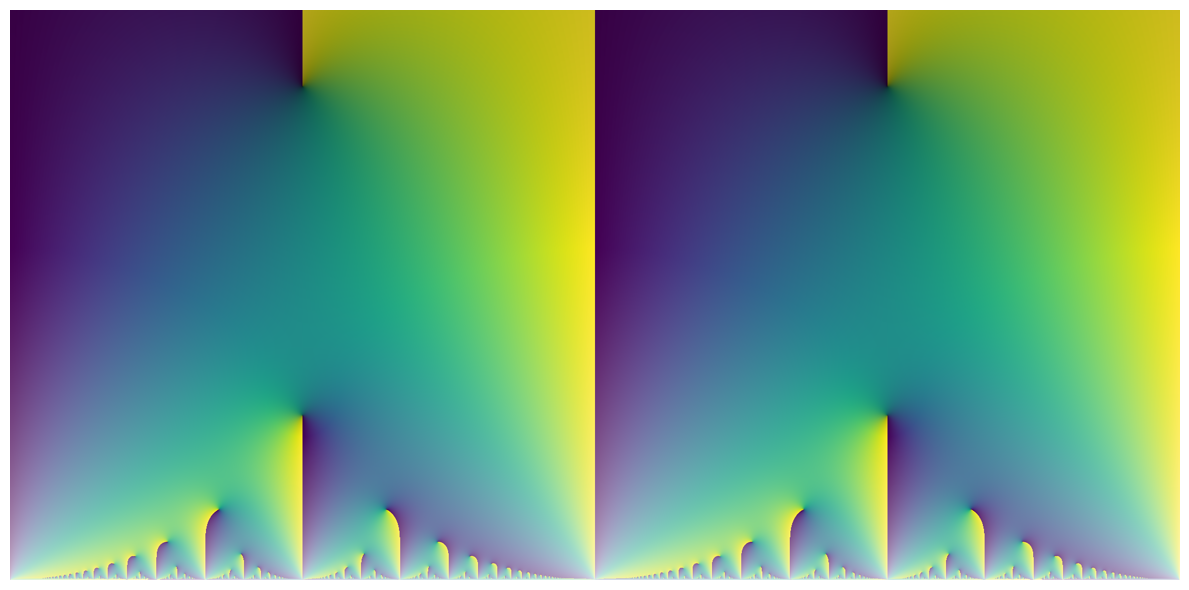}
\end{center}
\begin{center}
\begin{tikzpicture}[x=1in,y=1in]
\draw (-0.001,-0.085) node[anchor=south] {\includegraphics[width=4.067in,height=2.07in]{E4.png}};
\draw (2,0) arc(0:180:2);
\draw (0,0) arc(0:90:2);
\draw (0,0) arc(180:90:2);
\draw (1,0)--(1,2);
\draw (-1,0)--(-1,2);
\draw (-2/3,0) arc(0:180:{2/3}); \draw (0,0) arc(0:180:{2/3}); \draw (4/3,0) arc(0:180:{2/3}); \draw (2,0) arc(0:180:{2/3}); \draw (-6/5,0) arc(0:180:{2/5}); \draw (0,0) arc(0:180:{2/5}); \draw (4/5,0) arc(0:180:{2/5}); \draw (2,0) arc(0:180:{2/5}); \draw (-10/7,0) arc(0:180:{2/7}); \draw (0,0) arc(0:180:{2/7}); \draw (4/7,0) arc(0:180:{2/7}); \draw (2,0) arc(0:180:{2/7}); \draw (-1,0) arc(0:180:{1/4}); \draw (-1/2,0) arc(0:180:{1/4}); \draw (1,0) arc(0:180:{1/4}); \draw (3/2,0) arc(0:180:{1/4}); \draw (-14/9,0) arc(0:180:{2/9}); \draw (0,0) arc(0:180:{2/9}); \draw (4/9,0) arc(0:180:{2/9}); \draw (2,0) arc(0:180:{2/9}); \draw (-18/11,0) arc(0:180:{2/11}); \draw (0,0) arc(0:180:{2/11}); \draw (4/11,0) arc(0:180:{2/11}); \draw (2,0) arc(0:180:{2/11}); \draw (-22/13,0) arc(0:180:{2/13}); \draw (0,0) arc(0:180:{2/13}); \draw (4/13,0) arc(0:180:{2/13}); \draw (2,0) arc(0:180:{2/13}); \draw (-26/15,0) arc(0:180:{2/15}); \draw (-4/3,0) arc(0:180:{2/15}); \draw (-2/5,0) arc(0:180:{2/15}); \draw (0,0) arc(0:180:{2/15}); \draw (4/15,0) arc(0:180:{2/15}); \draw (2/3,0) arc(0:180:{2/15}); \draw (8/5,0) arc(0:180:{2/15}); \draw (2,0) arc(0:180:{2/15}); \draw (-1,0) arc(0:180:{1/8}); \draw (-3/4,0) arc(0:180:{1/8}); \draw (1,0) arc(0:180:{1/8}); \draw (5/4,0) arc(0:180:{1/8}); \draw (-30/17,0) arc(0:180:{2/17}); \draw (0,0) arc(0:180:{2/17}); \draw (4/17,0) arc(0:180:{2/17}); \draw (2,0) arc(0:180:{2/17}); \draw (-34/19,0) arc(0:180:{2/19}); \draw (0,0) arc(0:180:{2/19}); \draw (4/19,0) arc(0:180:{2/19}); \draw (2,0) arc(0:180:{2/19}); \draw (-38/21,0) arc(0:180:{2/21}); \draw (-8/7,0) arc(0:180:{2/21}); \draw (-2/3,0) arc(0:180:{2/21}); \draw (0,0) arc(0:180:{2/21}); \draw (4/21,0) arc(0:180:{2/21}); \draw (6/7,0) arc(0:180:{2/21}); \draw (4/3,0) arc(0:180:{2/21}); \draw (2,0) arc(0:180:{2/21}); \draw (-42/23,0) arc(0:180:{2/23}); \draw (0,0) arc(0:180:{2/23}); \draw (4/23,0) arc(0:180:{2/23}); \draw (2,0) arc(0:180:{2/23}); \draw (-3/2,0) arc(0:180:{1/12}); \draw (-1,0) arc(0:180:{1/12}); \draw (-5/6,0) arc(0:180:{1/12}); \draw (-1/3,0) arc(0:180:{1/12}); \draw (1/2,0) arc(0:180:{1/12}); \draw (1,0) arc(0:180:{1/12}); \draw (7/6,0) arc(0:180:{1/12}); \draw (5/3,0) arc(0:180:{1/12}); \draw (-46/25,0) arc(0:180:{2/25}); \draw (0,0) arc(0:180:{2/25}); \draw (4/25,0) arc(0:180:{2/25}); \draw (2,0) arc(0:180:{2/25}); \draw (-50/27,0) arc(0:180:{2/27}); \draw (0,0) arc(0:180:{2/27}); \draw (4/27,0) arc(0:180:{2/27}); \draw (2,0) arc(0:180:{2/27}); \draw (-54/29,0) arc(0:180:{2/29}); \draw (0,0) arc(0:180:{2/29}); \draw (4/29,0) arc(0:180:{2/29}); \draw (2,0) arc(0:180:{2/29}); \draw (-58/31,0) arc(0:180:{2/31}); \draw (0,0) arc(0:180:{2/31}); \draw (4/31,0) arc(0:180:{2/31}); \draw (2,0) arc(0:180:{2/31}); \draw (-1,0) arc(0:180:{1/16}); \draw (-7/8,0) arc(0:180:{1/16}); \draw (1,0) arc(0:180:{1/16}); \draw (9/8,0) arc(0:180:{1/16}); \draw (-62/33,0) arc(0:180:{2/33}); \draw (-4/3,0) arc(0:180:{2/33}); \draw (-6/11,0) arc(0:180:{2/33}); \draw (0,0) arc(0:180:{2/33}); \draw (4/33,0) arc(0:180:{2/33}); \draw (2/3,0) arc(0:180:{2/33}); \draw (16/11,0) arc(0:180:{2/33}); \draw (2,0) arc(0:180:{2/33}); \draw (-66/35,0) arc(0:180:{2/35}); \draw (-8/5,0) arc(0:180:{2/35}); \draw (-2/7,0) arc(0:180:{2/35}); \draw (0,0) arc(0:180:{2/35}); \draw (4/35,0) arc(0:180:{2/35}); \draw (2/5,0) arc(0:180:{2/35}); \draw (12/7,0) arc(0:180:{2/35}); \draw (2,0) arc(0:180:{2/35}); \draw (-70/37,0) arc(0:180:{2/37}); \draw (0,0) arc(0:180:{2/37}); \draw (4/37,0) arc(0:180:{2/37}); \draw (2,0) arc(0:180:{2/37}); \draw (-74/39,0) arc(0:180:{2/39}); \draw (-16/13,0) arc(0:180:{2/39}); \draw (-2/3,0) arc(0:180:{2/39}); \draw (0,0) arc(0:180:{2/39}); \draw (4/39,0) arc(0:180:{2/39}); \draw (10/13,0) arc(0:180:{2/39}); \draw (4/3,0) arc(0:180:{2/39}); \draw (2,0) arc(0:180:{2/39}); \draw (-7/5,0) arc(0:180:{1/20}); \draw (-1,0) arc(0:180:{1/20}); \draw (-9/10,0) arc(0:180:{1/20}); \draw (-1/2,0) arc(0:180:{1/20}); \draw (3/5,0) arc(0:180:{1/20}); \draw (1,0) arc(0:180:{1/20}); \draw (11/10,0) arc(0:180:{1/20}); \draw (3/2,0) arc(0:180:{1/20}); \draw (-78/41,0) arc(0:180:{2/41}); \draw (0,0) arc(0:180:{2/41}); \draw (4/41,0) arc(0:180:{2/41}); \draw (2,0) arc(0:180:{2/41}); \draw (-82/43,0) arc(0:180:{2/43}); \draw (0,0) arc(0:180:{2/43}); \draw (4/43,0) arc(0:180:{2/43}); \draw (2,0) arc(0:180:{2/43}); \draw (-86/45,0) arc(0:180:{2/45}); \draw (-10/9,0) arc(0:180:{2/45}); \draw (-4/5,0) arc(0:180:{2/45}); \draw (0,0) arc(0:180:{2/45}); \draw (4/45,0) arc(0:180:{2/45}); \draw (8/9,0) arc(0:180:{2/45}); \draw (6/5,0) arc(0:180:{2/45}); \draw (2,0) arc(0:180:{2/45}); \draw (-90/47,0) arc(0:180:{2/47}); \draw (0,0) arc(0:180:{2/47}); \draw (4/47,0) arc(0:180:{2/47}); \draw (2,0) arc(0:180:{2/47}); \draw (-5/3,0) arc(0:180:{1/24}); \draw (-1,0) arc(0:180:{1/24}); \draw (-11/12,0) arc(0:180:{1/24}); \draw (-1/4,0) arc(0:180:{1/24}); \draw (1/3,0) arc(0:180:{1/24}); \draw (1,0) arc(0:180:{1/24}); \draw (13/12,0) arc(0:180:{1/24}); \draw (7/4,0) arc(0:180:{1/24}); \draw (-94/49,0) arc(0:180:{2/49}); \draw (0,0) arc(0:180:{2/49}); \draw (4/49,0) arc(0:180:{2/49}); \draw (2,0) arc(0:180:{2/49}); \draw (-98/51,0) arc(0:180:{2/51}); \draw (-4/3,0) arc(0:180:{2/51}); \draw (-10/17,0) arc(0:180:{2/51}); \draw (0,0) arc(0:180:{2/51}); \draw (4/51,0) arc(0:180:{2/51}); \draw (2/3,0) arc(0:180:{2/51}); \draw (24/17,0) arc(0:180:{2/51}); \draw (2,0) arc(0:180:{2/51}); \draw (-102/53,0) arc(0:180:{2/53}); \draw (0,0) arc(0:180:{2/53}); \draw (4/53,0) arc(0:180:{2/53}); \draw (2,0) arc(0:180:{2/53}); \draw (-106/55,0) arc(0:180:{2/55}); \draw (-6/5,0) arc(0:180:{2/55}); \draw (-8/11,0) arc(0:180:{2/55}); \draw (0,0) arc(0:180:{2/55}); \draw (4/55,0) arc(0:180:{2/55}); \draw (4/5,0) arc(0:180:{2/55}); \draw (14/11,0) arc(0:180:{2/55}); \draw (2,0) arc(0:180:{2/55}); \draw (-3/2,0) arc(0:180:{1/28}); \draw (-1,0) arc(0:180:{1/28}); \draw (-13/14,0) arc(0:180:{1/28}); \draw (-3/7,0) arc(0:180:{1/28}); \draw (1/2,0) arc(0:180:{1/28}); \draw (1,0) arc(0:180:{1/28}); \draw (15/14,0) arc(0:180:{1/28}); \draw (11/7,0) arc(0:180:{1/28}); \draw (-110/57,0) arc(0:180:{2/57}); \draw (-24/19,0) arc(0:180:{2/57}); \draw (-2/3,0) arc(0:180:{2/57}); \draw (0,0) arc(0:180:{2/57}); \draw (4/57,0) arc(0:180:{2/57}); \draw (14/19,0) arc(0:180:{2/57}); \draw (4/3,0) arc(0:180:{2/57}); \draw (2,0) arc(0:180:{2/57}); \draw (-114/59,0) arc(0:180:{2/59}); \draw (0,0) arc(0:180:{2/59}); \draw (4/59,0) arc(0:180:{2/59}); \draw (2,0) arc(0:180:{2/59}); \draw (-118/61,0) arc(0:180:{2/61}); \draw (0,0) arc(0:180:{2/61}); \draw (4/61,0) arc(0:180:{2/61}); \draw (2,0) arc(0:180:{2/61}); \draw (-122/63,0) arc(0:180:{2/63}); \draw (-12/7,0) arc(0:180:{2/63}); \draw (-2/9,0) arc(0:180:{2/63}); \draw (0,0) arc(0:180:{2/63}); \draw (4/63,0) arc(0:180:{2/63}); \draw (2/7,0) arc(0:180:{2/63}); \draw (16/9,0) arc(0:180:{2/63}); \draw (2,0) arc(0:180:{2/63}); \draw (-1,0) arc(0:180:{1/32}); \draw (-15/16,0) arc(0:180:{1/32}); \draw (1,0) arc(0:180:{1/32}); \draw (17/16,0) arc(0:180:{1/32}); \draw (-126/65,0) arc(0:180:{2/65}); \draw (-20/13,0) arc(0:180:{2/65}); \draw (-2/5,0) arc(0:180:{2/65}); \draw (0,0) arc(0:180:{2/65}); \draw (4/65,0) arc(0:180:{2/65}); \draw (6/13,0) arc(0:180:{2/65}); \draw (8/5,0) arc(0:180:{2/65}); \draw (2,0) arc(0:180:{2/65}); \draw (-130/67,0) arc(0:180:{2/67}); \draw (0,0) arc(0:180:{2/67}); \draw (4/67,0) arc(0:180:{2/67}); \draw (2,0) arc(0:180:{2/67}); \draw (-134/69,0) arc(0:180:{2/69}); \draw (-4/3,0) arc(0:180:{2/69}); \draw (-14/23,0) arc(0:180:{2/69}); \draw (0,0) arc(0:180:{2/69}); \draw (4/69,0) arc(0:180:{2/69}); \draw (2/3,0) arc(0:180:{2/69}); \draw (32/23,0) arc(0:180:{2/69}); \draw (2,0) arc(0:180:{2/69}); \draw (-138/71,0) arc(0:180:{2/71}); \draw (0,0) arc(0:180:{2/71}); \draw (4/71,0) arc(0:180:{2/71}); \draw (2,0) arc(0:180:{2/71}); \draw (-13/9,0) arc(0:180:{1/36}); \draw (-1,0) arc(0:180:{1/36}); \draw (-17/18,0) arc(0:180:{1/36}); \draw (-1/2,0) arc(0:180:{1/36}); \draw (5/9,0) arc(0:180:{1/36}); \draw (1,0) arc(0:180:{1/36}); \draw (19/18,0) arc(0:180:{1/36}); \draw (3/2,0) arc(0:180:{1/36}); \draw (-142/73,0) arc(0:180:{2/73}); \draw (0,0) arc(0:180:{2/73}); \draw (4/73,0) arc(0:180:{2/73}); \draw (2,0) arc(0:180:{2/73}); \draw (-146/75,0) arc(0:180:{2/75}); \draw (-32/25,0) arc(0:180:{2/75}); \draw (-2/3,0) arc(0:180:{2/75}); \draw (0,0) arc(0:180:{2/75}); \draw (4/75,0) arc(0:180:{2/75}); \draw (18/25,0) arc(0:180:{2/75}); \draw (4/3,0) arc(0:180:{2/75}); \draw (2,0) arc(0:180:{2/75}); \draw (-150/77,0) arc(0:180:{2/77}); \draw (-12/11,0) arc(0:180:{2/77}); \draw (-6/7,0) arc(0:180:{2/77}); \draw (0,0) arc(0:180:{2/77}); \draw (4/77,0) arc(0:180:{2/77}); \draw (10/11,0) arc(0:180:{2/77}); \draw (8/7,0) arc(0:180:{2/77}); \draw (2,0) arc(0:180:{2/77}); \draw (-154/79,0) arc(0:180:{2/79}); \draw (0,0) arc(0:180:{2/79}); \draw (4/79,0) arc(0:180:{2/79}); \draw (2,0) arc(0:180:{2/79}); \draw (-7/4,0) arc(0:180:{1/40}); \draw (-1,0) arc(0:180:{1/40}); \draw (-19/20,0) arc(0:180:{1/40}); \draw (-1/5,0) arc(0:180:{1/40}); \draw (1/4,0) arc(0:180:{1/40}); \draw (1,0) arc(0:180:{1/40}); \draw (21/20,0) arc(0:180:{1/40}); \draw (9/5,0) arc(0:180:{1/40}); \draw (-158/81,0) arc(0:180:{2/81}); \draw (0,0) arc(0:180:{2/81}); \draw (4/81,0) arc(0:180:{2/81}); \draw (2,0) arc(0:180:{2/81}); \draw (-162/83,0) arc(0:180:{2/83}); \draw (0,0) arc(0:180:{2/83}); \draw (4/83,0) arc(0:180:{2/83}); \draw (2,0) arc(0:180:{2/83}); \draw (-166/85,0) arc(0:180:{2/85}); \draw (-8/5,0) arc(0:180:{2/85}); \draw (-6/17,0) arc(0:180:{2/85}); \draw (0,0) arc(0:180:{2/85}); \draw (4/85,0) arc(0:180:{2/85}); \draw (2/5,0) arc(0:180:{2/85}); \draw (28/17,0) arc(0:180:{2/85}); \draw (2,0) arc(0:180:{2/85}); \draw (-170/87,0) arc(0:180:{2/87}); \draw (-4/3,0) arc(0:180:{2/87}); \draw (-18/29,0) arc(0:180:{2/87}); \draw (0,0) arc(0:180:{2/87}); \draw (4/87,0) arc(0:180:{2/87}); \draw (2/3,0) arc(0:180:{2/87}); \draw (40/29,0) arc(0:180:{2/87}); \draw (2,0) arc(0:180:{2/87}); \draw (-3/2,0) arc(0:180:{1/44}); \draw (-1,0) arc(0:180:{1/44}); \draw (-21/22,0) arc(0:180:{1/44}); \draw (-5/11,0) arc(0:180:{1/44}); \draw (1/2,0) arc(0:180:{1/44}); \draw (1,0) arc(0:180:{1/44}); \draw (23/22,0) arc(0:180:{1/44}); \draw (17/11,0) arc(0:180:{1/44}); \draw (-174/89,0) arc(0:180:{2/89}); \draw (0,0) arc(0:180:{2/89}); \draw (4/89,0) arc(0:180:{2/89}); \draw (2,0) arc(0:180:{2/89}); \draw (-178/91,0) arc(0:180:{2/91}); \draw (-18/13,0) arc(0:180:{2/91}); \draw (-4/7,0) arc(0:180:{2/91}); \draw (0,0) arc(0:180:{2/91}); \draw (4/91,0) arc(0:180:{2/91}); \draw (8/13,0) arc(0:180:{2/91}); \draw (10/7,0) arc(0:180:{2/91}); \draw (2,0) arc(0:180:{2/91}); \draw (-182/93,0) arc(0:180:{2/93}); \draw (-40/31,0) arc(0:180:{2/93}); \draw (-2/3,0) arc(0:180:{2/93}); \draw (0,0) arc(0:180:{2/93}); \draw (4/93,0) arc(0:180:{2/93}); \draw (22/31,0) arc(0:180:{2/93}); \draw (4/3,0) arc(0:180:{2/93}); \draw (2,0) arc(0:180:{2/93}); \draw (-186/95,0) arc(0:180:{2/95}); \draw (-22/19,0) arc(0:180:{2/95}); \draw (-4/5,0) arc(0:180:{2/95}); \draw (0,0) arc(0:180:{2/95}); \draw (4/95,0) arc(0:180:{2/95}); \draw (16/19,0) arc(0:180:{2/95}); \draw (6/5,0) arc(0:180:{2/95}); \draw (2,0) arc(0:180:{2/95}); \draw (-13/8,0) arc(0:180:{1/48}); \draw (-1,0) arc(0:180:{1/48}); \draw (-23/24,0) arc(0:180:{1/48}); \draw (-1/3,0) arc(0:180:{1/48}); \draw (3/8,0) arc(0:180:{1/48}); \draw (1,0) arc(0:180:{1/48}); \draw (25/24,0) arc(0:180:{1/48}); \draw (5/3,0) arc(0:180:{1/48}); \draw (-190/97,0) arc(0:180:{2/97}); \draw (0,0) arc(0:180:{2/97}); \draw (4/97,0) arc(0:180:{2/97}); \draw (2,0) arc(0:180:{2/97}); \draw (-194/99,0) arc(0:180:{2/99}); \draw (-16/9,0) arc(0:180:{2/99}); \draw (-2/11,0) arc(0:180:{2/99}); \draw (0,0) arc(0:180:{2/99}); \draw (4/99,0) arc(0:180:{2/99}); \draw (2/9,0) arc(0:180:{2/99}); \draw (20/11,0) arc(0:180:{2/99}); \draw (2,0) arc(0:180:{2/99}); \draw (-198/101,0) arc(0:180:{2/101}); \draw (0,0) arc(0:180:{2/101}); \draw (4/101,0) arc(0:180:{2/101}); \draw (2,0) arc(0:180:{2/101}); \draw (-202/103,0) arc(0:180:{2/103}); \draw (0,0) arc(0:180:{2/103}); \draw (4/103,0) arc(0:180:{2/103}); \draw (2,0) arc(0:180:{2/103}); \draw (-19/13,0) arc(0:180:{1/52}); \draw (-1,0) arc(0:180:{1/52}); \draw (-25/26,0) arc(0:180:{1/52}); \draw (-1/2,0) arc(0:180:{1/52}); \draw (7/13,0) arc(0:180:{1/52}); \draw (1,0) arc(0:180:{1/52}); \draw (27/26,0) arc(0:180:{1/52}); \draw (3/2,0) arc(0:180:{1/52}); \draw (-206/105,0) arc(0:180:{2/105}); \draw (-10/7,0) arc(0:180:{2/105}); \draw (-4/3,0) arc(0:180:{2/105}); \draw (-6/5,0) arc(0:180:{2/105}); \draw (-16/21,0) arc(0:180:{2/105}); \draw (-22/35,0) arc(0:180:{2/105}); \draw (-8/15,0) arc(0:180:{2/105}); \draw (0,0) arc(0:180:{2/105}); \draw (4/105,0) arc(0:180:{2/105}); \draw (4/7,0) arc(0:180:{2/105}); \draw (2/3,0) arc(0:180:{2/105}); \draw (4/5,0) arc(0:180:{2/105}); \draw (26/21,0) arc(0:180:{2/105}); \draw (48/35,0) arc(0:180:{2/105}); \draw (22/15,0) arc(0:180:{2/105}); \draw (2,0) arc(0:180:{2/105}); \draw (-210/107,0) arc(0:180:{2/107}); \draw (0,0) arc(0:180:{2/107}); \draw (4/107,0) arc(0:180:{2/107}); \draw (2,0) arc(0:180:{2/107}); \draw (-214/109,0) arc(0:180:{2/109}); \draw (0,0) arc(0:180:{2/109}); \draw (4/109,0) arc(0:180:{2/109}); \draw (2,0) arc(0:180:{2/109}); \draw (-218/111,0) arc(0:180:{2/111}); \draw (-48/37,0) arc(0:180:{2/111}); \draw (-2/3,0) arc(0:180:{2/111}); \draw (0,0) arc(0:180:{2/111}); \draw (4/111,0) arc(0:180:{2/111}); \draw (26/37,0) arc(0:180:{2/111}); \draw (4/3,0) arc(0:180:{2/111}); \draw (2,0) arc(0:180:{2/111}); \draw (-5/4,0) arc(0:180:{1/56}); \draw (-1,0) arc(0:180:{1/56}); \draw (-27/28,0) arc(0:180:{1/56}); \draw (-5/7,0) arc(0:180:{1/56}); \draw (3/4,0) arc(0:180:{1/56}); \draw (1,0) arc(0:180:{1/56}); \draw (29/28,0) arc(0:180:{1/56}); \draw (9/7,0) arc(0:180:{1/56}); \draw (-222/113,0) arc(0:180:{2/113}); \draw (0,0) arc(0:180:{2/113}); \draw (4/113,0) arc(0:180:{2/113}); \draw (2,0) arc(0:180:{2/113}); \draw (-226/115,0) arc(0:180:{2/115}); \draw (-36/23,0) arc(0:180:{2/115}); \draw (-2/5,0) arc(0:180:{2/115}); \draw (0,0) arc(0:180:{2/115}); \draw (4/115,0) arc(0:180:{2/115}); \draw (10/23,0) arc(0:180:{2/115}); \draw (8/5,0) arc(0:180:{2/115}); \draw (2,0) arc(0:180:{2/115}); \draw (-230/117,0) arc(0:180:{2/117}); \draw (-14/13,0) arc(0:180:{2/117}); \draw (-8/9,0) arc(0:180:{2/117}); \draw (0,0) arc(0:180:{2/117}); \draw (4/117,0) arc(0:180:{2/117}); \draw (12/13,0) arc(0:180:{2/117}); \draw (10/9,0) arc(0:180:{2/117}); \draw (2,0) arc(0:180:{2/117}); \draw (-234/119,0) arc(0:180:{2/119}); \draw (-8/7,0) arc(0:180:{2/119}); \draw (-14/17,0) arc(0:180:{2/119}); \draw (0,0) arc(0:180:{2/119}); \draw (4/119,0) arc(0:180:{2/119}); \draw (6/7,0) arc(0:180:{2/119}); \draw (20/17,0) arc(0:180:{2/119}); \draw (2,0) arc(0:180:{2/119}); \draw (-9/5,0) arc(0:180:{1/60}); \draw (-5/3,0) arc(0:180:{1/60}); \draw (-3/2,0) arc(0:180:{1/60}); \draw (-1,0) arc(0:180:{1/60}); \draw (-29/30,0) arc(0:180:{1/60}); \draw (-7/15,0) arc(0:180:{1/60}); \draw (-3/10,0) arc(0:180:{1/60}); \draw (-1/6,0) arc(0:180:{1/60}); \draw (1/5,0) arc(0:180:{1/60}); \draw (1/3,0) arc(0:180:{1/60}); \draw (1/2,0) arc(0:180:{1/60}); \draw (1,0) arc(0:180:{1/60}); \draw (31/30,0) arc(0:180:{1/60}); \draw (23/15,0) arc(0:180:{1/60}); \draw (17/10,0) arc(0:180:{1/60}); \draw (11/6,0) arc(0:180:{1/60}); \draw (-238/121,0) arc(0:180:{2/121}); \draw (0,0) arc(0:180:{2/121}); \draw (4/121,0) arc(0:180:{2/121}); \draw (2,0) arc(0:180:{2/121}); \draw (-242/123,0) arc(0:180:{2/123}); \draw (-4/3,0) arc(0:180:{2/123}); \draw (-26/41,0) arc(0:180:{2/123}); \draw (0,0) arc(0:180:{2/123}); \draw (4/123,0) arc(0:180:{2/123}); \draw (2/3,0) arc(0:180:{2/123}); \draw (56/41,0) arc(0:180:{2/123}); \draw (2,0) arc(0:180:{2/123}); \draw (-246/125,0) arc(0:180:{2/125}); \draw (0,0) arc(0:180:{2/125}); \draw (4/125,0) arc(0:180:{2/125}); \draw (2,0) arc(0:180:{2/125}); \draw (-250/127,0) arc(0:180:{2/127}); \draw (0,0) arc(0:180:{2/127}); \draw (4/127,0) arc(0:180:{2/127}); \draw (2,0) arc(0:180:{2/127}); \draw (-1,0) arc(0:180:{1/64}); \draw (-31/32,0) arc(0:180:{1/64}); \draw (1,0) arc(0:180:{1/64}); \draw (33/32,0) arc(0:180:{1/64}); \draw (-254/129,0) arc(0:180:{2/129}); \draw (-56/43,0) arc(0:180:{2/129}); \draw (-2/3,0) arc(0:180:{2/129}); \draw (0,0) arc(0:180:{2/129}); \draw (4/129,0) arc(0:180:{2/129}); \draw (30/43,0) arc(0:180:{2/129}); \draw (4/3,0) arc(0:180:{2/129}); \draw (2,0) arc(0:180:{2/129}); \draw (-258/131,0) arc(0:180:{2/131}); \draw (0,0) arc(0:180:{2/131}); \draw (4/131,0) arc(0:180:{2/131}); \draw (2,0) arc(0:180:{2/131}); \draw (-262/133,0) arc(0:180:{2/133}); \draw (-32/19,0) arc(0:180:{2/133}); \draw (-2/7,0) arc(0:180:{2/133}); \draw (0,0) arc(0:180:{2/133}); \draw (4/133,0) arc(0:180:{2/133}); \draw (6/19,0) arc(0:180:{2/133}); \draw (12/7,0) arc(0:180:{2/133}); \draw (2,0) arc(0:180:{2/133}); \draw (-266/135,0) arc(0:180:{2/135}); \draw (-8/5,0) arc(0:180:{2/135}); \draw (-10/27,0) arc(0:180:{2/135}); \draw (0,0) arc(0:180:{2/135}); \draw (4/135,0) arc(0:180:{2/135}); \draw (2/5,0) arc(0:180:{2/135}); \draw (44/27,0) arc(0:180:{2/135}); \draw (2,0) arc(0:180:{2/135}); \draw (-25/17,0) arc(0:180:{1/68}); \draw (-1,0) arc(0:180:{1/68}); \draw (-33/34,0) arc(0:180:{1/68}); \draw (-1/2,0) arc(0:180:{1/68}); \draw (9/17,0) arc(0:180:{1/68}); \draw (1,0) arc(0:180:{1/68}); \draw (35/34,0) arc(0:180:{1/68}); \draw (3/2,0) arc(0:180:{1/68}); \draw (-270/137,0) arc(0:180:{2/137}); \draw (0,0) arc(0:180:{2/137}); \draw (4/137,0) arc(0:180:{2/137}); \draw (2,0) arc(0:180:{2/137}); \draw (-274/139,0) arc(0:180:{2/139}); \draw (0,0) arc(0:180:{2/139}); \draw (4/139,0) arc(0:180:{2/139}); \draw (2,0) arc(0:180:{2/139}); \draw (-278/141,0) arc(0:180:{2/141}); \draw (-4/3,0) arc(0:180:{2/141}); \draw (-30/47,0) arc(0:180:{2/141}); \draw (0,0) arc(0:180:{2/141}); \draw (4/141,0) arc(0:180:{2/141}); \draw (2/3,0) arc(0:180:{2/141}); \draw (64/47,0) arc(0:180:{2/141}); \draw (2,0) arc(0:180:{2/141}); \draw (-282/143,0) arc(0:180:{2/143}); \draw (-20/11,0) arc(0:180:{2/143}); \draw (-2/13,0) arc(0:180:{2/143}); \draw (0,0) arc(0:180:{2/143}); \draw (4/143,0) arc(0:180:{2/143}); \draw (2/11,0) arc(0:180:{2/143}); \draw (24/13,0) arc(0:180:{2/143}); \draw (2,0) arc(0:180:{2/143}); \draw (-11/9,0) arc(0:180:{1/72}); \draw (-1,0) arc(0:180:{1/72}); \draw (-35/36,0) arc(0:180:{1/72}); \draw (-3/4,0) arc(0:180:{1/72}); \draw (7/9,0) arc(0:180:{1/72}); \draw (1,0) arc(0:180:{1/72}); \draw (37/36,0) arc(0:180:{1/72}); \draw (5/4,0) arc(0:180:{1/72}); \draw (-286/145,0) arc(0:180:{2/145}); \draw (-34/29,0) arc(0:180:{2/145}); \draw (-4/5,0) arc(0:180:{2/145}); \draw (0,0) arc(0:180:{2/145}); \draw (4/145,0) arc(0:180:{2/145}); \draw (24/29,0) arc(0:180:{2/145}); \draw (6/5,0) arc(0:180:{2/145}); \draw (2,0) arc(0:180:{2/145}); \draw (-290/147,0) arc(0:180:{2/147}); \draw (-64/49,0) arc(0:180:{2/147}); \draw (-2/3,0) arc(0:180:{2/147}); \draw (0,0) arc(0:180:{2/147}); \draw (4/147,0) arc(0:180:{2/147}); \draw (34/49,0) arc(0:180:{2/147}); \draw (4/3,0) arc(0:180:{2/147}); \draw (2,0) arc(0:180:{2/147}); \draw (-294/149,0) arc(0:180:{2/149}); \draw (0,0) arc(0:180:{2/149}); \draw (4/149,0) arc(0:180:{2/149}); \draw (2,0) arc(0:180:{2/149}); \draw (-298/151,0) arc(0:180:{2/151}); \draw (0,0) arc(0:180:{2/151}); \draw (4/151,0) arc(0:180:{2/151}); \draw (2,0) arc(0:180:{2/151}); \draw (-3/2,0) arc(0:180:{1/76}); \draw (-1,0) arc(0:180:{1/76}); \draw (-37/38,0) arc(0:180:{1/76}); \draw (-9/19,0) arc(0:180:{1/76}); \draw (1/2,0) arc(0:180:{1/76}); \draw (1,0) arc(0:180:{1/76}); \draw (39/38,0) arc(0:180:{1/76}); \draw (29/19,0) arc(0:180:{1/76}); \draw (-302/153,0) arc(0:180:{2/153}); \draw (-26/17,0) arc(0:180:{2/153}); \draw (-4/9,0) arc(0:180:{2/153}); \draw (0,0) arc(0:180:{2/153}); \draw (4/153,0) arc(0:180:{2/153}); \draw (8/17,0) arc(0:180:{2/153}); \draw (14/9,0) arc(0:180:{2/153}); \draw (2,0) arc(0:180:{2/153}); \draw (-306/155,0) arc(0:180:{2/155}); \draw (-6/5,0) arc(0:180:{2/155}); \draw (-24/31,0) arc(0:180:{2/155}); \draw (0,0) arc(0:180:{2/155}); \draw (4/155,0) arc(0:180:{2/155}); \draw (4/5,0) arc(0:180:{2/155}); \draw (38/31,0) arc(0:180:{2/155}); \draw (2,0) arc(0:180:{2/155}); \draw (-310/157,0) arc(0:180:{2/157}); \draw (0,0) arc(0:180:{2/157}); \draw (4/157,0) arc(0:180:{2/157}); \draw (2,0) arc(0:180:{2/157}); \draw (-314/159,0) arc(0:180:{2/159}); \draw (-4/3,0) arc(0:180:{2/159}); \draw (-34/53,0) arc(0:180:{2/159}); \draw (0,0) arc(0:180:{2/159}); \draw (4/159,0) arc(0:180:{2/159}); \draw (2/3,0) arc(0:180:{2/159}); \draw (72/53,0) arc(0:180:{2/159}); \draw (2,0) arc(0:180:{2/159}); \draw (-11/8,0) arc(0:180:{1/80}); \draw (-1,0) arc(0:180:{1/80}); \draw (-39/40,0) arc(0:180:{1/80}); \draw (-3/5,0) arc(0:180:{1/80}); \draw (5/8,0) arc(0:180:{1/80}); \draw (1,0) arc(0:180:{1/80}); \draw (41/40,0) arc(0:180:{1/80}); \draw (7/5,0) arc(0:180:{1/80}); \draw (-318/161,0) arc(0:180:{2/161}); \draw (-12/7,0) arc(0:180:{2/161}); \draw (-6/23,0) arc(0:180:{2/161}); \draw (0,0) arc(0:180:{2/161}); \draw (4/161,0) arc(0:180:{2/161}); \draw (2/7,0) arc(0:180:{2/161}); \draw (40/23,0) arc(0:180:{2/161}); \draw (2,0) arc(0:180:{2/161}); \draw (-322/163,0) arc(0:180:{2/163}); \draw (0,0) arc(0:180:{2/163}); \draw (4/163,0) arc(0:180:{2/163}); \draw (2,0) arc(0:180:{2/163}); \draw (-326/165,0) arc(0:180:{2/165}); \draw (-52/33,0) arc(0:180:{2/165}); \draw (-72/55,0) arc(0:180:{2/165}); \draw (-16/15,0) arc(0:180:{2/165}); \draw (-10/11,0) arc(0:180:{2/165}); \draw (-2/3,0) arc(0:180:{2/165}); \draw (-2/5,0) arc(0:180:{2/165}); \draw (0,0) arc(0:180:{2/165}); \draw (4/165,0) arc(0:180:{2/165}); \draw (14/33,0) arc(0:180:{2/165}); \draw (38/55,0) arc(0:180:{2/165}); \draw (14/15,0) arc(0:180:{2/165}); \draw (12/11,0) arc(0:180:{2/165}); \draw (4/3,0) arc(0:180:{2/165}); \draw (8/5,0) arc(0:180:{2/165}); \draw (2,0) arc(0:180:{2/165}); \draw (-330/167,0) arc(0:180:{2/167}); \draw (0,0) arc(0:180:{2/167}); \draw (4/167,0) arc(0:180:{2/167}); \draw (2,0) arc(0:180:{2/167}); \draw (-11/6,0) arc(0:180:{1/84}); \draw (-23/14,0) arc(0:180:{1/84}); \draw (-31/21,0) arc(0:180:{1/84}); \draw (-1,0) arc(0:180:{1/84}); \draw (-41/42,0) arc(0:180:{1/84}); \draw (-1/2,0) arc(0:180:{1/84}); \draw (-1/3,0) arc(0:180:{1/84}); \draw (-1/7,0) arc(0:180:{1/84}); \draw (1/6,0) arc(0:180:{1/84}); \draw (5/14,0) arc(0:180:{1/84}); \draw (11/21,0) arc(0:180:{1/84}); \draw (1,0) arc(0:180:{1/84}); \draw (43/42,0) arc(0:180:{1/84}); \draw (3/2,0) arc(0:180:{1/84}); \draw (5/3,0) arc(0:180:{1/84}); \draw (13/7,0) arc(0:180:{1/84}); \draw (-334/169,0) arc(0:180:{2/169}); \draw (0,0) arc(0:180:{2/169}); \draw (4/169,0) arc(0:180:{2/169}); \draw (2,0) arc(0:180:{2/169}); \draw (-338/171,0) arc(0:180:{2/171}); \draw (-14/9,0) arc(0:180:{2/171}); \draw (-8/19,0) arc(0:180:{2/171}); \draw (0,0) arc(0:180:{2/171}); \draw (4/171,0) arc(0:180:{2/171}); \draw (4/9,0) arc(0:180:{2/171}); \draw (30/19,0) arc(0:180:{2/171}); \draw (2,0) arc(0:180:{2/171}); \draw (-342/173,0) arc(0:180:{2/173}); \draw (0,0) arc(0:180:{2/173}); \draw (4/173,0) arc(0:180:{2/173}); \draw (2,0) arc(0:180:{2/173}); \draw (-346/175,0) arc(0:180:{2/175}); \draw (-28/25,0) arc(0:180:{2/175}); \draw (-6/7,0) arc(0:180:{2/175}); \draw (0,0) arc(0:180:{2/175}); \draw (4/175,0) arc(0:180:{2/175}); \draw (22/25,0) arc(0:180:{2/175}); \draw (8/7,0) arc(0:180:{2/175}); \draw (2,0) arc(0:180:{2/175}); \draw (-19/11,0) arc(0:180:{1/88}); \draw (-1,0) arc(0:180:{1/88}); \draw (-43/44,0) arc(0:180:{1/88}); \draw (-1/4,0) arc(0:180:{1/88}); \draw (3/11,0) arc(0:180:{1/88}); \draw (1,0) arc(0:180:{1/88}); \draw (45/44,0) arc(0:180:{1/88}); \draw (7/4,0) arc(0:180:{1/88}); \draw (-350/177,0) arc(0:180:{2/177}); \draw (-4/3,0) arc(0:180:{2/177}); \draw (-38/59,0) arc(0:180:{2/177}); \draw (0,0) arc(0:180:{2/177}); \draw (4/177,0) arc(0:180:{2/177}); \draw (2/3,0) arc(0:180:{2/177}); \draw (80/59,0) arc(0:180:{2/177}); \draw (2,0) arc(0:180:{2/177}); \draw (-354/179,0) arc(0:180:{2/179}); \draw (0,0) arc(0:180:{2/179}); \draw (4/179,0) arc(0:180:{2/179}); \draw (2,0) arc(0:180:{2/179}); \draw (-358/181,0) arc(0:180:{2/181}); \draw (0,0) arc(0:180:{2/181}); \draw (4/181,0) arc(0:180:{2/181}); \draw (2,0) arc(0:180:{2/181}); \draw (-362/183,0) arc(0:180:{2/183}); \draw (-80/61,0) arc(0:180:{2/183}); \draw (-2/3,0) arc(0:180:{2/183}); \draw (0,0) arc(0:180:{2/183}); \draw (4/183,0) arc(0:180:{2/183}); \draw (42/61,0) arc(0:180:{2/183}); \draw (4/3,0) arc(0:180:{2/183}); \draw (2,0) arc(0:180:{2/183}); \draw (-3/2,0) arc(0:180:{1/92}); \draw (-1,0) arc(0:180:{1/92}); \draw (-45/46,0) arc(0:180:{1/92}); \draw (-11/23,0) arc(0:180:{1/92}); \draw (1/2,0) arc(0:180:{1/92}); \draw (1,0) arc(0:180:{1/92}); \draw (47/46,0) arc(0:180:{1/92}); \draw (35/23,0) arc(0:180:{1/92}); \draw (-366/185,0) arc(0:180:{2/185}); \draw (-8/5,0) arc(0:180:{2/185}); \draw (-14/37,0) arc(0:180:{2/185}); \draw (0,0) arc(0:180:{2/185}); \draw (4/185,0) arc(0:180:{2/185}); \draw (2/5,0) arc(0:180:{2/185}); \draw (60/37,0) arc(0:180:{2/185}); \draw (2,0) arc(0:180:{2/185}); \draw (-370/187,0) arc(0:180:{2/187}); \draw (-14/11,0) arc(0:180:{2/187}); \draw (-12/17,0) arc(0:180:{2/187}); \draw (0,0) arc(0:180:{2/187}); \draw (4/187,0) arc(0:180:{2/187}); \draw (8/11,0) arc(0:180:{2/187}); \draw (22/17,0) arc(0:180:{2/187}); \draw (2,0) arc(0:180:{2/187}); \draw (-374/189,0) arc(0:180:{2/189}); \draw (-38/27,0) arc(0:180:{2/189}); \draw (-4/7,0) arc(0:180:{2/189}); \draw (0,0) arc(0:180:{2/189}); \draw (4/189,0) arc(0:180:{2/189}); \draw (16/27,0) arc(0:180:{2/189}); \draw (10/7,0) arc(0:180:{2/189}); \draw (2,0) arc(0:180:{2/189}); \draw (-378/191,0) arc(0:180:{2/191}); \draw (0,0) arc(0:180:{2/191}); \draw (4/191,0) arc(0:180:{2/191}); \draw (2,0) arc(0:180:{2/191}); \draw (-5/3,0) arc(0:180:{1/96}); \draw (-1,0) arc(0:180:{1/96}); \draw (-47/48,0) arc(0:180:{1/96}); \draw (-5/16,0) arc(0:180:{1/96}); \draw (1/3,0) arc(0:180:{1/96}); \draw (1,0) arc(0:180:{1/96}); \draw (49/48,0) arc(0:180:{1/96}); \draw (27/16,0) arc(0:180:{1/96}); \draw (-382/193,0) arc(0:180:{2/193}); \draw (0,0) arc(0:180:{2/193}); \draw (4/193,0) arc(0:180:{2/193}); \draw (2,0) arc(0:180:{2/193}); \draw (-386/195,0) arc(0:180:{2/195}); \draw (-24/13,0) arc(0:180:{2/195}); \draw (-4/3,0) arc(0:180:{2/195}); \draw (-46/39,0) arc(0:180:{2/195}); \draw (-4/5,0) arc(0:180:{2/195}); \draw (-42/65,0) arc(0:180:{2/195}); \draw (-2/15,0) arc(0:180:{2/195}); \draw (0,0) arc(0:180:{2/195}); \draw (4/195,0) arc(0:180:{2/195}); \draw (2/13,0) arc(0:180:{2/195}); \draw (2/3,0) arc(0:180:{2/195}); \draw (32/39,0) arc(0:180:{2/195}); \draw (6/5,0) arc(0:180:{2/195}); \draw (88/65,0) arc(0:180:{2/195}); \draw (28/15,0) arc(0:180:{2/195}); \draw (2,0) arc(0:180:{2/195}); \draw (-390/197,0) arc(0:180:{2/197}); \draw (0,0) arc(0:180:{2/197}); \draw (4/197,0) arc(0:180:{2/197}); \draw (2,0) arc(0:180:{2/197}); \draw (-394/199,0) arc(0:180:{2/199}); \draw (0,0) arc(0:180:{2/199}); \draw (4/199,0) arc(0:180:{2/199}); \draw (2,0) arc(0:180:{2/199}); \draw (-37/25,0) arc(0:180:{1/100}); \draw (-1,0) arc(0:180:{1/100}); \draw (-49/50,0) arc(0:180:{1/100}); \draw (-1/2,0) arc(0:180:{1/100}); \draw (13/25,0) arc(0:180:{1/100}); \draw (1,0) arc(0:180:{1/100}); \draw (51/50,0) arc(0:180:{1/100}); \draw (3/2,0) arc(0:180:{1/100}); 
\end{tikzpicture}
\end{center}
\caption{\\A plot of the Eisenstein series $E_4(z)$ for $-1 \le \Re z \le 1$ and $0 < \Im z \le 1$ (above) and the same plot overlaid with a tiling of $\uhp$ using fundamental domains for the action of $\SL_2(\Z)$ (below).}
\label{figure:E4}
\end{figure}

The mappings $z \mapsto z+1$ and $z \mapsto -1/z$ that occur in these functional equations generate a discrete group of linear fractional transforms of the \emph{upper half-plane} $\uhp = \{z \in \CC : \Im z > 0\}$. To put it into a broader context of matrix groups, we can let the matrix $\bigl(\begin{smallmatrix} a & b \\ c & d \end{smallmatrix}\bigr)$ act on $\uhp$ via
\[
\begin{pmatrix}a & b \\ c & d \end{pmatrix} \cdot z = \frac{az+b}{cz+d}.
\]
Then the matrices $T = \bigl(\begin{smallmatrix} 1 & 1 \\ 0 & 1 \end{smallmatrix}\bigr)$ and $S = \bigl(\begin{smallmatrix} 0 & -1 \\ 1 & 0 \end{smallmatrix}\bigr)$ satisfy $T \cdot z = z+1$ and $S \cdot z = -1/z$, and they turn out to generate the group $\SL_2(\Z)$.

The \emph{weight $k$ action} of $\SL_2(\Z)$ on functions $f \colon \uhp \to \CC$ is defined by
\[
(f |_k \gamma) (z) = (cz+d)^{-k} f\bigg(\frac{az+b}{cz+d}\bigg)
\]
for $\gamma = \bigl(\begin{smallmatrix} a & b \\ c & d \end{smallmatrix}\bigr)$. In this notation, the functional equations $E_k(z+1)=E_k(z)$ and $E_k(-1/z)=z^k E_k(z)$ imply that the Eisenstein series $E_k$ satisfies $E_k |_k \gamma = E_k$ for all $\gamma \in \SL_2(\Z)$ when $k>2$.

A \emph{modular form of weight $k$ for $\SL_2(\Z)$} is a holomorphic function $f \colon \uhp \to \CC$ such that $f |_k \gamma = f$ for all $\gamma \in \SL_2(\Z)$ and one additional condition holds, called being holomorphic at infinity. To state this condition, note that taking $\gamma=T$ shows that $f(z+1)=f(z)$, and thus we can expand $f$ as a Fourier series
\[
f(z) = \sum_{n \in \Z} a_n e^{2\pi i n z}.
\]
We say $f$ is \emph{meromorphic at infinity} if there are only finitely many nonzero coefficients $a_n$ with $n<0$, and \emph{holomorphic at infinity} if $a_n=0$ for all $n<0$. The name reflects the fact that this Fourier series governs the behavior of $f(z)$ as $\Im z$ grows, because $e^{2\pi i z} \to 0$ as $\Im z \to \infty$. The Fourier series of a modular form is often known as its \emph{$q$-series}, with~$q = e^{2\pi i z}$.

The normalization factor $1/(2\zeta(k))$ in \eqref{eq:eisenstein} ensures that the $q$-series of $E_k$ has rational coefficients, and even integral coefficients when $k$ is small. For example, one can show that $E_4(z) = 1 + 240\sum_{n \ge 1} \sigma_3(n) q^n$ and $E_6(z) = 1-504\sum_{n \ge 1} \sigma_5(n) q^n$, where $\sigma_k(n)$ denotes the sum of the $k$-th powers of the divisors of $n$.

The product of modular forms of weights $k$ and $\ell$ is a modular form of weight $k+\ell$, and modular forms therefore form a graded ring. For $\SL_2(\Z)$, one can show that this ring is generated by $E_4$ and $E_6$. In other words, the vector space of modular forms of weight $k$ for $\SL_2(Z)$ is spanned by the modular forms $E_4^j E_6^\ell$ with $4j+6\ell=k$.

In addition to using Eisenstein series directly, Viazovska also uses
the \emph{modular discriminant} $\Delta$, which is given by
\begin{equation} \label{eq:Delta}
\Delta(z) = \frac{E_4(z)^3 - E_6(z)^2}{1728} = q \prod_{n=1}^\infty (1-q^n)^{24}.
\end{equation}
Its key property is that it vanishes nowhere in the upper half plane, while it vanishes at infinity (in the sense that its $q$-series has no constant term).

Tur\'an said that special functions should instead be called useful functions, and modular forms are no exception to this principle. The reason we study modular forms is not that we have a special love for Eisenstein series, but rather that the functional equations $f(z+1)=f(z)$ and $f(-1/z) = z^k f(z)$ arise far more often than one might expect. For example, the $E_8$ lattice has an important modular form associated with it, namely its \emph{theta series}
\[
\Theta_{E_8}(z) = \sum_{n=0}^\infty N_n e^{2\pi i n z},
\]
where $N_n = \#\{x \in E_8 : |x|^2 = 2n\}$. In other words, the theta series is a generating function that counts the number of vectors of each length in $E_8$.

This theta series satisfies both functional equations: 
$\Theta_{E_8}(z+1) = \Theta_{E_8}(z)$ follows from the definition of $\Theta_{E_8}$ as a Fourier series, while $\Theta_{E_8}(-1/z) = z^4 \Theta_{E_8}$ amounts to Poisson summation over $E_8$ for the complex Gaussian $x \mapsto e^{\pi i z |x|^2}$, which has eight-dimensional Fourier transform $y \mapsto z^{-4} e^{\pi i (-1/z) |y|^2}$. These functional equations tell us that $\Theta_{E_8}$ is a modular form for $\SL_2(\Z)$ of weight $4$, and it must therefore be proportional to $E_4$. In fact, $\Theta_{E_8}=E_4$, because $N_0=1$. Thus, we obtain the beautiful formula $240 \sigma_3(n)$ for the number of vectors in $E_8$ of squared norm $2n$.

The theory of modular forms extends to other discrete groups, if one carefully defines what being holomorphic at infinity means.\footnote{If $\Gamma$ is a subgroup of finite index in $\SL_2(\Z)$, then the condition is that for each $\gamma \in \SL_2(\Z)$, $f|_k\gamma$ should be holomorphic at infinity. Note that $f|_k\gamma$ need not satisfy $(f|_k\gamma)(z+1) = (f|_k\gamma)(z)$, but one can check that it always satisfies $(f|_k\gamma)(z+n) = (f|_k\gamma)(z)$ for some positive integer $n$ and thus has a Fourier expansion in $e^{2\pi i z/n} = q^{1/n}$.} Viazovska's proof makes use of one more group, namely
\[
\Gamma(2) = \Bigg\{\gamma \in \SL_2(\Z) : \gamma \equiv \begin{pmatrix}1& 0 \\ 0 & 1 \end{pmatrix} \pmod{2}\Bigg\},
\]
which has index~$6$ in $\SL_2(\Z)$.
If we let 
\[
U(z) = \bigg(\sum_{n\in \Z} e^{\pi i n^2 z}\bigg)^4,
\]
$W = U|_2 T$, and $V = U-W$, then $U$, $V$, and $W$ are modular forms of weight $2$ for $\Gamma(2)$ that satisfy $U=V+W$ and
\begin{equation}\label{eq:UVWslash}
\begin{aligned}
  U |_2T &= W,  & V |_2T &= -V, & W |_2T &= U,  \\
  U |_2S &= -U, & V |_2S &= -W, & W |_2S &= -V.
\end{aligned}
\end{equation}
These identities will play a key role in the construction of Viazovska's magic function. It turns out that $U$ and $W$ generate the ring of modular forms for $\Gamma(2)$,
and therefore every modular form of weight $2k$ for $\Gamma(2)$ is a linear combination of $U^k$, $U^{k-1}W$, $U^{k-2}W^2$, \dots, $W^k$.

Because modular forms are so closely connected with lattices, it is natural to turn to modular forms when attempting to construct the magic functions. However, it is entirely unclear where we should even start, because modular forms are completely different sorts of objects from radial Schwartz functions. Figure~\ref{figure:E4} looks nothing whatsoever like Figures~\ref{figure:diagram} or~\ref{figure:magicf}, and there is no familiar transformation that makes it look any more similar.

\section{Viazovska's construction for single roots}

The first step in Viazovska's construction of the magic function $f$ is to split $f$ into eigenfunctions of the Fourier transform. Radial functions satisfy $\widehat{\widehat{f}\,\,} = f$, and so we can write $f$ as $f = f_+ + f_-$, where $f_+ := (f+\widehat{f})/2$ satisfies $\widehat{f_+}=f_+$ and $f_- := (f-\widehat{f})/2$ satisfies $\widehat{f_-}=-f_-$.  If $f$ is the magic function in eight dimensions, then $f$ and $\widehat{f}$ both have roots at $\sqrt{2n}$ for integers $n \ge 1$, and therefore $f_+$ and $f_-$ do as well. Thus, we are looking for radial Fourier eigenfunctions with specified roots. Specifically, each of $f_\pm$ should have a single root at $\sqrt{2}$ and double roots at $\sqrt{2n}$ for $n \ge 2$. These roots turn out to provide enough information to determine $f_\pm$ up to scaling, and they can then be combined to obtain $f$.

Before we construct the actual magic function, it is worth examining a simpler variant as a warm-up exercise. Instead of trying to control the behavior of $f$ to second order at $\sqrt{2n}$, we will instead control the behavior of a function $g$ to first order at $\sqrt{n}$.
This construction has no known applications to sphere packing, but it is nevertheless of intrinsic interest in Fourier analysis. We will also focus on the $-1$ eigenfunction (i.e., the case $\widehat{g}=-g$) in the single-root case, for the sake of specificity.

Viazovska found a remarkable integral transform that can construct such functions.
We will write a radial function $g \colon \R^8 \to \CC$ as a continuous linear combination of complex Gaussians $x \mapsto e^{\pi i z |x|^2}$ with $z \in \uhp$ via the contour integral
\begin{equation} \label{eq:fcontour}
g(x) = \tfrac{1}{2}\int_{-1}^1 \psi(z) e^{\pi i z |x|^2} \, dz,
\end{equation}
where $\psi$ is a holomorphic function on $\uhp$ and the contour is a semicircle centered at the origin. Under which conditions on $\psi$ will $g$ be a Fourier eigenfunction, and how can we control its values at $\sqrt{n}$?

We can obtain the values $g\big(\sqrt{n}\big)$ by imposing periodicity on $\psi$ as follows.
Suppose $\psi(z+2) = \psi(z)$ for all $z \in \uhp$, so that $\psi$ has a Fourier series of the form
\begin{equation} \label{eq:Fourierg}
\psi(z) = \sum_{n \in \Z} a_n e^{\pi i n z}.
\end{equation}
Then for integers $n \ge 0$,
\[
g\big(\sqrt{n}\big) = \tfrac{1}{2}\int_{-1}^1 \psi(z) e^{\pi i n z} \, dz = a_{-n}
\]
by orthogonality, provided that we can interchange the sum and integral. If the Fourier expansion \eqref{eq:Fourierg} has only finitely many negative terms, then $g\big(\sqrt{n}\big)$ will vanish for all but finitely many $n$.

To compute the Fourier transform of $g$, we can interchange the contour integral and Fourier transform, again assuming the integral is sufficiently well behaved. Then
\[
\widehat{g}(y) = \tfrac{1}{2}\int_{-1}^1 \psi(z) z^{-4} e^{\pi i (-1/z) |y|^2} \, dz,
\]
because the $d$-dimensional Fourier transform of the complex Gaussian $x \mapsto e^{\pi i z |x|^2}$ with $z \in \uhp$ is given by $y \mapsto (i/z)^{d/2} e^{\pi i (-1/z) |y|^2}$, and $d=8$ here. Changing variables to $u = -1/z$ shows that
\[
\widehat{g}(y) = -\tfrac{1}{2}\int_{-1}^1 \psi(-1/u) u^{2} e^{\pi i u |y|^2} \, du.
\]
In other words, taking the Fourier transform of $g$ amounts to replacing $\psi$ with $- \psi |_{-2} S$, and we obtain $\widehat{g} = -g$ if $\psi|_{-2}S = \psi$.

Let $\Gamma$ be the subgroup of $\SL_2(\Z)$ generated by $S$ and $T^2$, which has index~$3$ in $\SL_2(\Z)$. Then the conditions that $\psi|_{-2} T^2 = \psi$ (i.e., $\psi(z+2)=\psi(z)$) and $\psi|_{-2} S = \psi$ mean that $\psi$ is \emph{weakly modular of weight $-2$} for $\Gamma$. The reason why $\psi$ is less than a full-fledged modular form is that it is only meromorphic at infinity (this is unavoidable, since the weight is negative). We furthermore require $\psi$ to vanish at $\pm 1$, which will be enough to justify our integral manipulations and show that $g$ is a Schwartz function. In terms of Fourier series, this vanishing says that $\psi|_{-2} TS$ has no negative terms in its $q$-series, because $TS$ maps the cusp $i\infty$ to $1$.

We will construct an example of the form $\psi = \psi_0/\Delta$ using the $\Delta$ function from \eqref{eq:Delta}, where $\psi_0$ is a genuine modular form of weight $10$ for $\Gamma$. Note that the denominator of $\Delta$ causes no difficulties in $\uhp$, since $\Delta(z) \ne 0$ for all $z \in \uhp$, and the zero of $\Delta$ at infinity will lead to a pole of $\psi$.

The function $\psi_0$ is modular of weight $10$ for $\Gamma$, and thus also for $\Gamma(2)$ because $\Gamma(2)$ is a subgroup of $\Gamma$. In particular, $\psi_0$ must be a linear combination of $U^5$, $U^4W$, $U^3W^2$, \dots, $W^5$, because $U$ and $W$ generate the ring of modular forms for $\Gamma(2)$. The relations \eqref{eq:UVWslash} specify the action of $S$ and $T$, and they imply that the subspace invariant under $S$ is spanned by
\begin{align*}
\alpha &:= U^5 - 6U^3W^2 + 4U^2W^3,\\
 \beta &:= U^4W - 3U^3W^2 + 2U^2W^3,\text{ and}\\
\gamma &:= -U^3W^2 + 4U^2W^3 - 5UW^4 + 2W^5,
\end{align*}
with $q$-expansions
\begin{align*}
\frac{\alpha}{\Delta} & = -q^{-1} - 40q^{-1/2} +752 + \cdots, & \frac{\alpha}{\Delta}\Big|_{-2} TS &= -1024 + 90112q + \cdots,\\
 \frac{\beta}{\Delta} &= -16q^{-1/2} + 256 + \cdots, &  \frac{\beta}{\Delta}\Big|_{-2} TS &= -512 - 20480q + \cdots,\\
 \frac{\gamma}{\Delta} &= 256 - 10240q^{1/2} + \cdots, & \frac{\gamma}{\Delta}\Big|_{-2} TS &= -2q^{-1} - 32 + \cdots
\end{align*}
in terms of $q^{1/2} = e^{\pi i z}$. Now  requiring $\psi$ to vanish at $\pm 1$ determines it up to scaling as 
\begin{equation} \label{eq:gFourier}
\psi = \frac{2\beta-\alpha}{\Delta} = q^{-1} + 8q^{-1/2} - 240 - 6176 q^{1/2} - \cdots,
\end{equation}
which yields a radial Schwartz function $g \colon \R^8 \to \R$ such that $\widehat{g}=-g$ and
\[
g\big(\sqrt{n}\big) = \begin{cases} -240 & \text{if $n=0$,}\\
8 & \text{if $n=1$,}\\
1 & \text{if $n=2$, and}\\
0 & \text{if $n\ge 3$.}\\
\end{cases}
\]
Note that we do not have much flexibility here: the values $g(0)$, $g(1)$, and $g\big(\sqrt{2}\big)$ are uniquely determined by Poisson summation over $\Z^8$ and $E_8$, up to scaling.

We can rewrite the definition of $f$ in another useful form as follows. If $|x|$ is large enough (in fact, $|x|^2>2$ will suffice), then
\begin{align*}
g(x) &= \tfrac{1}{2}\int_{-1}^1 \psi(z) e^{\pi i z |x|^2} \, dz\\
&= \tfrac{1}{2}\int_{-1}^i \psi(z) e^{\pi i z |x|^2} \, dz - \tfrac{1}{2}\int_{1}^i \psi(z) e^{\pi i z |x|^2} \, dz\\
&= \tfrac{1}{2}\int_{-1}^{-1+i\infty} \psi(z) e^{\pi i z |x|^2} \, dz - \tfrac{1}{2}\int_{1}^{1+i\infty} \psi(z) e^{\pi i z |x|^2} \, dz\\
&= \frac{e^{-\pi i |x|^2} - e^{\pi i |x|^2}}{2} \int_0^{i\infty} \psi(u+1) e^{\pi i u |x|^2}\, du.
\end{align*}
In these manipulations, the second line merely breaks the integral in two, the third line uses the fact that
\[
\int_{-1 + iR}^{1+iR} \psi(z) e^{\pi i z |x|^2} \, dz \to 0
\]
as $R \to \infty$ (which holds if $|x|^2$ is large enough), and the fourth line uses $\psi(u-1)=\psi(u+1)$.

In other words, $g(x)$ is given by $\sin(\pi |x|^2)$ times the Laplace transform of $t \mapsto \psi(it+1)$ evaluated at $\pi |x|^2$:
\begin{equation} \label{eq:singleLaplace}
g(x) = \sin (\pi |x|^2) \int_0^{\infty} \psi(it+1) e^{-\pi t |x|^2}\, dt.
\end{equation}
While the original integral \eqref{eq:fcontour} converges for all $x$, this integral converges only when $|x|^2$ is large enough
for the Gaussian factor $e^{-\pi t |x|^2}$ to counteract the growth of $\psi(it+1)$ as $t \to \infty$. In particular, \eqref{eq:gFourier} implies that
\[
\psi(it+1) = e^{2\pi t} - 8e^{\pi t} - 240 + 6176 e^{-\pi t} - \cdots
\]
as $t \to \infty$, which means we need $|x|^2>2$.
We can use this expansion to analytically continue $g$ by removing the divergent terms:
\begin{align*}
g(x) &= \sin(\pi |x|^2) \int_0^\infty (e^{2\pi t} - 8e^{\pi t} - 240)e^{-\pi t |x|^2}\, dt\\
& \quad \phantom{} + \sin (\pi |x|^2) \int_0^{\infty} (\psi(it+1)-e^{2\pi t} + 8e^{\pi t} + 240) e^{-\pi t |x|^2}\, dt\\
&= \frac{\sin(\pi |x|^2)}{\pi(|x|^2-2)}- \frac{8\sin(\pi |x|^2)}{\pi(|x|^2-1)} - \frac{240 \sin(\pi |x|^2)}{\pi |x|^2} \\
& \quad \phantom{} + \sin (\pi |x|^2) \int_0^{\infty} (\psi(it+1)-e^{2\pi t} + 8e^{\pi t} + 240) e^{-\pi t |x|^2}\, dt,
\end{align*}
and this last formula holds regardless of $|x|$, with removable singularities at $|x| = 0$, $1$, and $\sqrt{2}$.

\section{Viazovska's construction for double roots}

We are now in a position to obtain the magic function in eight dimensions. First, we will obtain the $-1$ eigenfunction $f_-$.
It is not immediately clear how to generalize the contour integral \eqref{eq:fcontour} from single to double roots, but the Laplace transform formula \eqref{eq:singleLaplace} generalizes elegantly. To obtain $f_-$,
we will look for a special function $\psi$ such that
\[
f_-(x) = - 4 i \sin (\pi |x|^2/2)^2 \int_0^{i\infty} \psi(z) e^{\pi i z |x|^2 } \, dz
\]
when $|x|$ is large enough. If we write $-4 \sin (\pi |x|^2/2)^2 = e^{-\pi i |x|^2} + e^{\pi i |x|^2} - 2$, we find that
\begin{align*}
f_-(x) &= \int_{-1}^{-1+i\infty} \psi(z+1) e^{\pi i |x|^2 z} \, dz +  \int_{1}^{1+i\infty} \psi(z-1) e^{\pi i |x|^2 z} \, dz\\
& \quad \phantom{}- 2  \int_{0}^{i\infty} \psi(z) e^{\pi i |x|^2 z} \, dz.
\end{align*}
We will construct a function $\psi$ such that $\psi$ is holomorphic on $\uhp$ and $\psi(z)$ is exponentially bounded as $\Im z \to \infty$. Under these conditions, when $|x|$ is sufficiently large
we can shift the contours and combine the integrals to obtain
\begin{align*}
f_-(x) &= \int_{-1}^i \psi(z+1) e^{\pi i |x|^2 z} \, dz + \int_1^i \psi(z-1) e^{\pi i |x|^2 z} \, dz\\
& \quad \phantom{} - 2 \int_0^i \psi(z) e^{\pi i |x|^2 z} \, dz + \int_i^{i\infty} \big(\psi(z+1) + \psi(z-1) - 2\psi(z)\big) e^{\pi i |x|^2 z} \, dz,
\end{align*}
with the contours shown in Figure~\ref{figure:fourcontours}.
This formula will be the analogue of \eqref{eq:fcontour}, and it will define $f_-(x)$ for all $x$.

\begin{figure}[t]
\begin{center}
\begin{tikzpicture}
\draw[white] (-4,0)--(4,0); 
\draw (0,0)--(0,4);
\draw[-stealth] (2,0) arc(0:45:2);
\draw[-stealth] (-2,0) arc(180:135:2);
\draw (2,0) arc(0:180:2);
\fill (0,0) circle (0.05);
\draw (0,0) node[below] {$0$};
\fill (2,0) circle (0.05);
\draw (2,0) node[below] {$1$};
\fill (-2,0) circle (0.05);
\draw (-2,0) node[below] {$-1$};
\fill (0,2) circle (0.05);
\draw (0,2) node[below right] {$i$};
\draw[-stealth] (0,0)--(0,1);
\draw[-stealth] (0,2)--(0,3);
\draw (0,1) node[left] {$-2 \psi(z)$};
\draw (0,3) node[left] {$\psi(z+1) + \psi(z-1) - 2\psi(z)$};
\draw (-1.4142135,1.4142135) node[above left] {$\psi(z+1)$};
\draw (1.4142135,1.4142135) node[above right] {$\psi(z-1)$};
\draw (0,4) node[above] {$\vdots$};
\end{tikzpicture}
\end{center}
\caption{\\The contours used to obtain $f_-(x)$, labeled with their integrands (omitting $e^{\pi i |x|^2 z} \, dz$).}
\label{figure:fourcontours}
\end{figure}

Taking the Fourier transform amounts to replacing $e^{\pi i |x|^2 z}$ with $z^{-4} e^{\pi i |y|^2 (-1/z)}$ in the formula defining $f_-$:
\begin{align*}
\widehat{f_-}(y) &= \int_{-1}^i \psi(z+1) z^{-4} e^{\pi i |y|^2 (-1/z)} \, dz + \int_1^i \psi(z-1) z^{-4} e^{\pi i |y|^2 (-1/z)} \, dz\\
& \quad \phantom{} - 2 \int_0^i \psi(z) z^{-4} e^{\pi i |y|^2 (-1/z)} \, dz\\
& \quad \phantom{} + \int_i^{i\infty} \big(\psi(z+1) + \psi(z-1) - 2\psi(z)\big) z^{-4} e^{\pi i |y|^2 (-1/z)} \, dz.
\end{align*}
We can now set $u = -1/z$, which exchanges the four contours in pairs. The simplest way to obtain $\widehat{f_-} = -f_-$
would be if the resulting formula is exactly the negative of the formula with which we began. That amounts to the functional
equations
\begin{flalign*}
&& \psi |_{-2} TS &= - \psi |_{-2} T^{-1} &&\\
&\text{and}\hidewidth\\
&& 2\psi |_{-2} S &= 2\psi - \psi |_{-2} T - \psi|_{-2} T^{-1}. &&
\end{flalign*}
Note that the structure of these equations reflects the integrands.

Now the question is which sorts of functions $\psi$ satisfy these functional equations.
The simplest possibility would be some sort of  modular form. The functional equations are not consistent with invariance under $S$ and $T$, and so $\psi$ cannot be modular for the full group $\SL_2(\Z)$. Let us suppose instead that $\psi$ is weakly modular of weight $-2$ for $\Gamma(2)$ (i.e., invariant under $\Gamma(2)$ but only meromorphic at infinity).
Then $\psi |_{-2}T = \psi |_{-2} T^{-1}$, because $T^2 \in \Gamma(2)$, and our functional equations become $\psi |_{-2}  TS = - \psi |_{-2} T$ and $\psi = \psi |_{-2}  T + \psi |_{-2}  S$.
Furthermore, the second equation implies the first, because $S^2 = I$. We will therefore obtain the eigenfunction equation $\widehat{f_-} = -f_-$ as long as $\psi$ is weakly modular of weight $-2$ for $\Gamma(2)$ and satisfies $\psi = \psi |_{-2}  T + \psi |_{-2}  S$.

As in the single-root case, it is natural to multiply $\psi$ by $\Delta$ to try to eliminate a pole at infinity. Then $\psi \Delta$ will be a genuine modular form of weight $10$ for $\Gamma(2)$, and thus a linear combination of $U^5$, $U^4W$, $U^3W^2$, \dots, $W^5$. One can check that the solutions of the remaining functional equation form a two-dimensional subspace, spanned by
\[
\alpha := 2U^4W - 4U^3W^2 + U^2W^3 +UW^4 \quad\text{and}\quad \beta := 5U^4W - 10U^3W^2 + 5U^2W^3 + W^5,
\]
with
\[
\frac{\alpha}{\Delta} = -16q^{-1/2} + 768 + \cdots \quad\text{and}\quad \frac{\beta}{\Delta} = q^{-1} - 40 q^{-1/2} + 2064 + \cdots .
\]
We will take
\[
\psi = \frac{-5\alpha + 2\beta}{\Delta} = 2 q^{-1} + 288 + \cdots,
\]
so that we eliminate the $q^{-1/2}$ term in the $q$-series. The motivation for eliminating that term is that it prevents $f_-$ from having a pole at radius~$1$. To see why,
let us analytically continue
\[
f_-(x) = 4  \sin (\pi |x|^2/2)^2 \int_0^{\infty} \psi(it) e^{-\pi t |x|^2 } \, dt
\]
as in the single-root case. If $\psi(it) = a_2 e^{2\pi t} + a_1 e^{\pi t} + a_0 + \cdots$ as $t \to \infty$, then
\begin{align*}
f_-(x) &= \frac{4a_2\sin (\pi |x|^2/2)^2}{\pi(|x|^2-2)}- \frac{4a_1\sin (\pi |x|^2/2)^2}{\pi(|x|^2-1)} - \frac{4a_0\sin (\pi |x|^2/2)^2}{\pi |x|^2} \\
& \quad \phantom{} + 4\sin (\pi |x|^2/2)^2 \int_0^{\infty} (\psi(it) - a_2 e^{2\pi t} - a_1 e^{\pi t} - a_0 ) e^{-\pi t |x|^2}\, dt.
\end{align*}
Here the $a_1$ term has a pole unless $a_1=0$. For our choice of $\psi$, $(a_2,a_1,a_0) = (2,0,288)$, and thus $f_-$ has a single root at $\sqrt{2}$ and
double roots at $\sqrt{2n}$ for $n \ge 2$. One can also check that $\psi(it)$ vanishes as $t \to 0+$ (equivalently, $\psi |_{-2} S$ vanishes at infinity), which is
enough for $f_-$ to be a Schwartz function and to justify all our integral manipulations.

We have therefore obtained a magic eigenfunction $f_-$ as
\[
f_-(x) = 4 \sin (\pi |x|^2/2)^2 \int_0^{\infty} \psi(it) e^{-\pi t |x|^2 } \, dt
\]
for $|x|^2 > 2$, where
\begin{equation} \label{eq:psidef}
\psi = \frac{W^3(5U^2 - 5UW + 2W^2)}{\Delta}.
\end{equation}
Our scaling here does not yet match the magic function for sphere packing, but aside from that we have exactly what we need.

Equation \eqref{eq:psidef} implies that $\psi(it) > 0$ for all $t \in (0,\infty)$. (Specifically, $\Delta(it)>0$ thanks to its product formula, $W(it)>0$ since it is the fourth power of a real quantity,
and $5U(it)^2 - 5U(it)W(it) + 2W(it)^2 > 0$ since it is a positive-definite quadratic form.) It follows that $f_-$ never changes sign beyond radius $\sqrt{2}$, in accordance with our expectations. However, note that our eigenfunction is positive beyond radius $\sqrt{2}$,
and so we will have to correct its sign later to match the magic function.

All that remains is to construct a magic eigenfunction $f_+$ and take a suitable linear combination of $f_+$ and $f_-$ to obtain $f$. Constructing $f_+$ is very much like constructing $f_-$.
If we define $f_+$ for $|x|$ sufficiently large by
\[
f_+(x) = - 4 i \sin (\pi |x|^2/2)^2 \int_0^{i\infty} \phi(z) e^{\pi i z |x|^2 } \, dz
\]
for some holomorphic function $\phi \colon \uhp \to \CC$, then the eigenfunction equation $\widehat{f_+} = f_+$ will follow from the functional equations
\begin{flalign*}
&& \phi |_{-2} TS &= \phi |_{-2} T^{-1} &&\\
&\text{and}\hidewidth\\
&& 2\phi |_{-2} S &= -2\phi + \phi |_{-2} T + \phi|_{-2} T^{-1}. &&
\end{flalign*}
These are the same functional equations as we required for $\psi$, except for a factor of $-1$.

A little manipulation using $(ST)^3=I$ shows that the first functional equation is equivalent to $\phi |_{-2} ST = \phi|_{-2} S$. Thus, if we set $\chi := \phi |_{-2}S$, then $\chi$ must be invariant under $T$.
However, the second functional equation is more subtle. A short calculation shows that if $\chi |_{0} S = \chi$ (equivalently, $(\chi |_{-2} S)(z) = z^2 \chi(z)$), then the second functional equation holds. In other words, it is enough for $\chi$ to be weakly
modular of weight $0$ for $\SL_2(\Z)$. However, such functions turn out not to be sufficient to obtain $f_+$. If one tries to solve for undetermined coefficients to construct $f_+$, as in the $f_-$ case, one finds that there is no solution with the needed properties.

Instead, we can use \emph{quasimodular forms}, not just modular forms. Recall that the Eisenstein series $E_2$ was not a modular form of weight $2$, because conditional convergence interfered with the series manipulations needed to prove modularity. If we let
\[
E_2(z) = 1-24 \sum_{n \ge 1} \sigma_1(n) q^n,
\]
then $E_2$ turns out to satisfy
\[
z^{-2} E_2(-1/z) = E_2(z) - \frac{6i}{\pi z},
\]
with the $6i/(\pi z)$ term amounting to the deviation from modularity. A \emph{quasimodular form of weight $k$ and depth $\ell$} for $\SL_2(\Z)$ is a sum $f_k + f_{k-2}E_2 + \dots + f_{k-\ell} E_2^\ell$, where each $f_j$ is a modular form of weight $k-2j$.

Instead of just a weakly modular form of weight $0$, one can check that the function $\chi$ can be a weakly quasimodular form of weight $0$ and depth~$2$ for $\SL_2(\Z)$.
Now we have enough flexibility to construct $f_+$, and calculations much like those in the $f_-$ case lead to
\[
\chi = \frac{(E_2E_4-E_6)^2}{\Delta},
\]
up to scaling. See Figure~\ref{figure:Viazovskamodular} for plots of the quasimodular forms that yield $f_-$ and $f_+$.

\begin{figure}[t]
\begin{center}
\includegraphics[width=4.07in,height=2.07in]{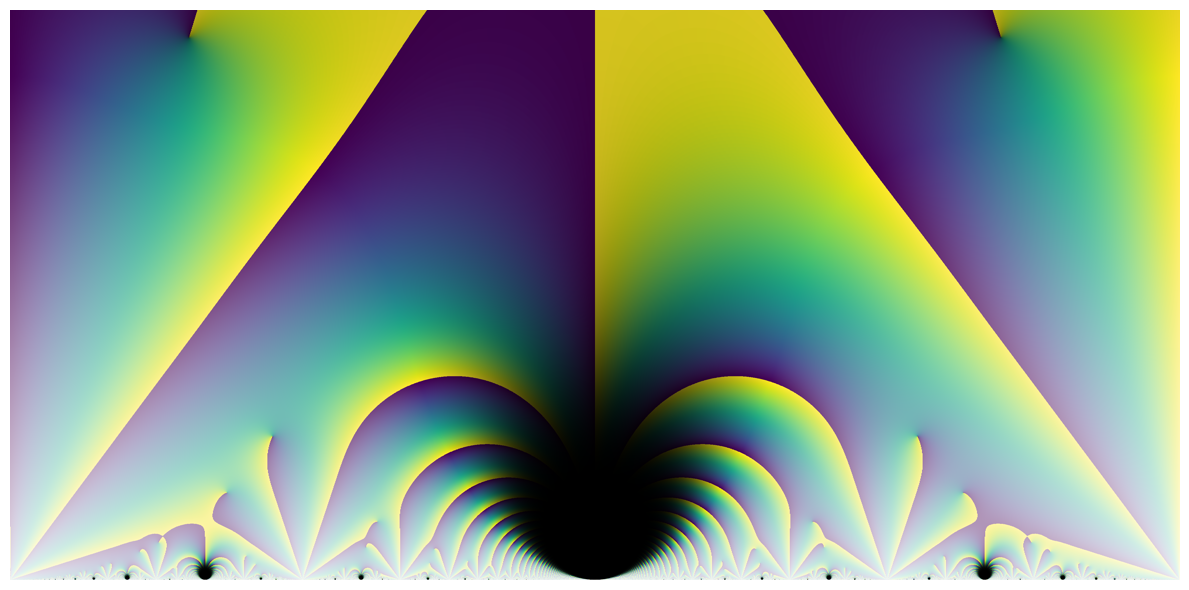}
\end{center}
\begin{center}
\includegraphics[width=4.07in,height=2.07in]{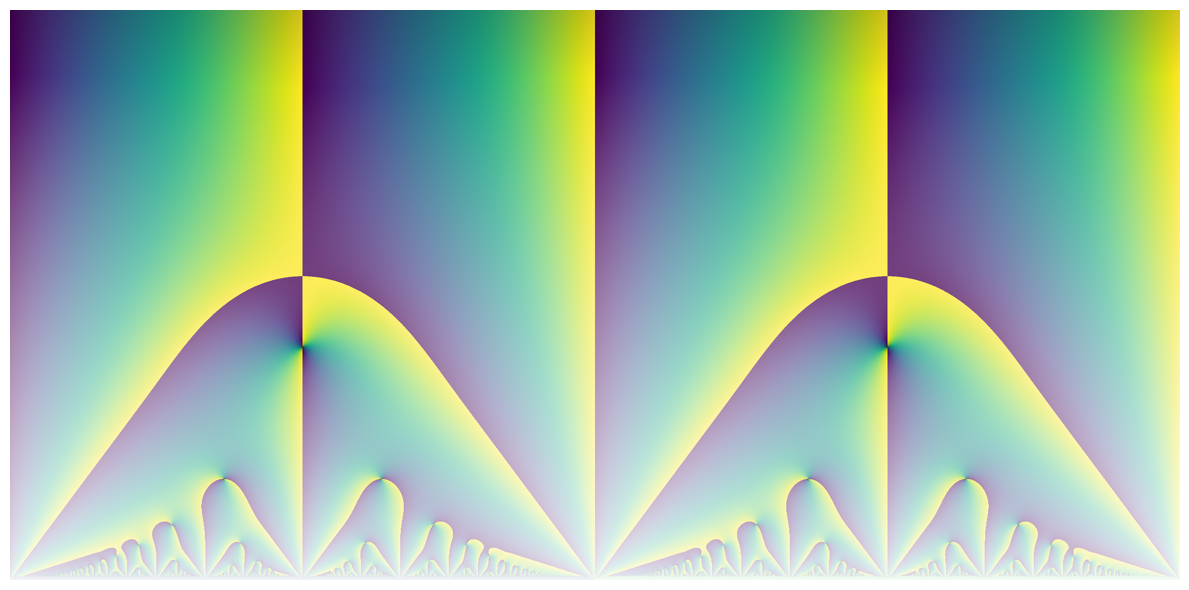}
\end{center}
\caption{\\Plots of $\psi(z)\Delta(z)$  (above) and $(\varphi|_{-2}S)(z)\Delta(z)$ (below) for $-1 \le \Re z \le 1$ and $0 < \Im z \le 1$.}
\label{figure:Viazovskamodular}
\end{figure}

Now that we have obtained both magic eigenfunctions, we can construct the magic function $f$ as a linear combination of them. First, we rescale $\phi$ so that $f_+(0)=1$, and then we rescale $\psi$ so that
$f_-'\big(\sqrt{2}\big) = f_+'\big(\sqrt{2}\big)$, to obtain a double root at $\sqrt{2}$ for $\widehat{f}$.
Using these scalings, the eight-dimensional magic function is given by
\[
f(x) = 4 \sin(\pi |x|^2/2)^2 \int_0^\infty (\phi(it) + \psi(it)) e^{-\pi t |x|^2} \, dt
\]
for $|x|^2 > 2$, and the eigenfunction property implies that
\[
\widehat{f}(y) = 4 \sin(\pi |y|^2/2)^2 \int_0^\infty (\phi(it) - \psi(it)) e^{-\pi t |y|^2} \, dt
\]
for all $y \ne 0$ (this integral turns out to converge whenever $|y|>0$, because the exponential growth in $\phi(it)$ and $\psi(it)$ as $t \to \infty$ cancels).

The final step in the proof of Theorem~\ref{thm:dim8} is to check the inequalities that are needed for Theorem~\ref{thm:LPbound}, namely $f(x) \le 0$ for $|x| \ge 2$ and $\widehat{f}(y) \ge 0$ for all $y$, to make sure there are no unexpected sign changes between the roots $\sqrt{2n}$. In principle that might seem difficult, because integral transforms of quasimodular forms could be complicated. However, these inequalities hold for the simplest reason one could hope for:
\[
\phi(it) + \psi(it) < 0 \quad\text{and}\quad \phi(it) - \psi(it) > 0
\]
for all $t > 0$. In other words, the desired inequalities hold directly at the level of the quasimodular forms themselves. This can be checked rigorously in any of several ways. For example, one can use asymptotics to check the inequalities as $t \to 0$ or $t \to \infty$, and then use interval arithmetic to verify them on the remaining bounded interval.

Overall, this proof feels like a miracle. Everything falls beautifully into place, with Viazovska's constructions having just enough flexibility to complete the proof in a unique way. What I find most impressive is the number of ingenious ideas required for the full proof. The single-root construction is itself remarkable, generalizing it to $f_-$ is even more so, and still more ideas are required for $f_+$. Viazovska is a master of special functions, whose work would surely have excited Jacobi and Ramanujan.

\section{Interpolation and consequences}

Along the way to proving the optimality of $E_8$, Viazovska made the bold conjecture that the magic function is uniquely determined by its required roots, and that more generally a radial Schwartz function on $\R^8$ is uniquely determined by its values and radial derivatives at the radii $\sqrt{2n}$ and those of its Fourier transform. It is far from obvious that it is possible in principle to reconstruct a radial Schwartz function from discrete data of this sort.

Radchenko and Viazovska took a major step in this direction by proving a one-dimensional analogue for first-order interpolation, and the second-order theorem was proved by Cohn, Kumar, Miller, Radchenko, and Viazovska.

\begin{theorem}[Radchenko and Viazovska \cite{RV2019}] There exist Schwartz functions $a_n \colon \R \to \R$ such that for every Schwartz function $f \colon \R \to \R$ and $x \in \R$,
\[
f(x) = \sum_{n \in \Z} f\big(\sqrt{n}\big) \, a_n(x) + \sum_{n \in \Z} \widehat{f}\big(\sqrt{n}\big) \, \widehat{a}_n(x).
\]
\end{theorem}

\begin{theorem}[Cohn, Kumar, Miller, Radchenko, and Viazovska \cite{CKMRV2022}] \label{theorem:interpolation}
Let $(d,n_0)$ be $(8,1)$ or $(24,2)$. Then every radial Schwartz function $f \colon \R^d \to \R$ is uniquely determined by the values
$f\big(\sqrt{2n}\big)$, $f'\big(\sqrt{2n}\big)$,
$\widehat{f}\big(\sqrt{2n}\big)$, and $\widehat{f}\,'\big(\sqrt{2n}\big)$ for
integers $n \ge n_0$.  Specifically, there exists an interpolation basis
$a_n, b_n $
for $n \ge n_0$ such that for every radial Schwartz function $f$ and $x
\in \R^d$,
\[
\begin{split}
f(x) &= \sum_{n=n_0}^\infty f\big(\sqrt{2n}\big)\,a_n(x)+\sum_{n=n_0}^\infty
f'\big(\sqrt{2n}\big)\,b_n(x)\\
&\quad\phantom{}+\sum_{n=n_0}^\infty \widehat{f}\big(\sqrt{2n}\big)\,\widehat{a}_n(x)
+\sum_{n=n_0}^\infty \widehat{f}\,'\big(\sqrt{2n}\big)\,\widehat{b}_n(x).
\end{split}     
\]
\end{theorem}

The proofs construct the interpolation bases explicitly, by combining Viazovska's integral transform techniques with broader classes of special functions.

One consequence of radial Fourier interpolation is a stronger optimality theorem for $E_8$ and the Leech lattice. Instead of just taking into account local interactions between particles,
as in the sphere packing problem, one can study optimization problems with long-range interactions. For example, one could ask for the ground state of particles interacting via
an inverse power law. Cohn and Kumar \cite{CK2007} formulated a broad notion of optimality, called \emph{universal optimality}, and radial Fourier interpolation yields corresponding magic functions:

\begin{theorem}[Cohn, Kumar, Miller, Radchenko, and Viazovska \cite{CKMRV2022}]  \label{theorem:univopt}
The $E_8$ root lattice and the Leech lattice are universally optimal in
$\R^8$ and $\R^{24}$, respectively.
\end{theorem}

\section{The future}

Although Viazovska's work has settled several major questions, much remains to be understood. For example, the theory of interpolation for radial Schwartz functions is rapidly developing, with noteworthy connections to uniqueness theory for the Klein-Gordon equation \cite{BHMRV2021}.

One puzzling issue is two dimensions. While the two-dimensional sphere packing problem can be settled by elementary geometry, universal optimality remains a tantalizing conjecture. There seems to be a magic function for $d=2$ in Theorem~\ref{thm:LPbound}, with $r = (4/3)^{1/4}$; no proof is known, but numerical computations agree with the optimal packing density in $\R^2$ to over one thousand decimal places. Furthermore, analogous magic functions seem to exist for universal optimality in $\R^2$. However, it is unclear what sort of function space might allow a suitable interpolation theory (see Section~7 in \cite{CKMRV2022}).

There are also remarkable connections with conformal field theory and quantum gravity \cite{HMR2019}. When $d$ is even, the linear programming bound for the sphere packing density in $\R^d$ turns out to be equivalent to the spinless modular bootstrap bound for the spectral gap in a theory of $d/2$ free bosons, and the conformal bootstrap program generalizes it to a family of related bounds. How these more general bounds might relate to discrete geometry remains a mystery.



\begin{thebibliography}{99}
\bibitem{ACHLT2020} N.~Afkhami-Jeddi, H.~Cohn, T.~Hartman, D.~de Laat, and
    A.~Tajdini,
    \emph{High-dimensional sphere packing and the modular bootstrap},
    J.\ High Energy Phys.\ \textbf{2020} (2020), no.\ 12, Paper No.\ 066, 44 pp.
    \arXiv{2006.02560} \\ 
     \mydoi{10.1007/jhep12(2020)066}

\bibitem{BHMRV2021} A.~Bakan, H.~Hedenmalm, A.~Montes-Rodríguez, D.~Radchenko, and M.~Viazovska, 
\emph{Fourier uniqueness in even dimensions}, 
Proc.\ Natl.\ Acad.\ Sci.\ USA \textbf{118} (2021), no.\ 15, Paper No. 2023227118, 4 pp.  \mydoi{10.1073/pnas.2023227118}

\bibitem{BRV2013} A.~Bondarenko, D.~Radchenko, and M.~Viazovska,
\emph{Optimal asymptotic bounds for spherical designs},
Ann.\ of Math.\ (2) \textbf{178} (2013), no.\ 2, 443--452. \arXiv{1009.4407} \mydoi{10.4007/annals.2013.178.2.2}

\bibitem{BCK2010} J.~Bourgain, L.~Clozel, and J.-P.~Kahane,
\emph{Principe d'Heisenberg et fonctions positives},
Ann.\ Inst.\ Fourier (Grenoble) \textbf{60} (2010), no.\ 4, 1215--1232. \arXiv{0811.4360} \mydoi{10.5802/aif.2552 }

\bibitem{CS2017} H.~Cohen and F.~Str\"omberg,
\emph{Modular Forms: A Classical Approach}, Graduate Studies in Mathematics \textbf{179}, American Mathematical Society, Providence, RI, 2017. 

\bibitem{C2017} H.~Cohn,
\emph{A conceptual breakthrough in sphere packing},
Notices Amer.\ Math.\ Soc.\ \textbf{64} (2017), no.\ 2, 102--115. \arXiv{1611.01685} \mydoi{10.1090/noti1474}

\bibitem{CE2003} H.~Cohn and N.~Elkies,  \emph{New
    upper bounds on sphere packings I}, Ann.\ of Math.\ (2) \textbf{157} (2003), no.\ 2,
    689--714. \arXiv{math/0110009} \\ 
    \mydoi{10.4007/annals.2003.157.689}
   
\bibitem{CG2019} H.~Cohn and F.~Gon\c{c}alves, 
\emph{An optimal uncertainty principle in twelve dimensions via modular forms},
Invent.\ Math.\ \textbf{217} (2019), no.\ 3, 799--831. \arXiv{1712.04438} \mydoi{10.1007/s00222-019-00875-4}

\bibitem{CK2007} H.~Cohn and A.~Kumar, \emph{Universally optimal
    distribution of points on spheres}, J.\ Amer.\ Math.\ Soc.\ \textbf{20} (2007), no.\ 1,
    99--148. \arXiv{math/0607446} \\ 
    \mydoi{10.1090/S0894-0347-06-00546-7} 

\bibitem{CKMRV2017} H.~Cohn, A.~Kumar, S.~D.~Miller, D.~Radchenko, and
    M.~Viazovska, \emph{The sphere packing problem in dimension $24$},
    Ann.\ of Math.\ (2) \textbf{185} (2017), no.\ 3,
    1017--1033. \arXiv{1603.06518} \mydoi{10.4007/annals.2017.185.3.8}
    
\bibitem{CKMRV2022} H.~Cohn, A.~Kumar, S.~D.~Miller, D.~Radchenko, and
    M.~Viazovska, \emph{Universal optimality of the $E_8$ and Leech lattices
    and interpolation formulas}, Annals of Mathematics, to appear. \arXiv{1902.05438}

\bibitem{CS1999} J.~H.~Conway and N.~J.~A.~Sloane, \emph{Sphere
    Packings, Lattices and Groups}, third edition, Grundlehren der Mathematischen
    Wissenschaften \textbf{290}, Springer-Verlag, New York, 1999. \mydoi{10.1007/978-1-4757-6568-7}

\bibitem{D1972} P.~Delsarte, \emph{Bounds for unrestricted codes, by linear programming},
Philips Res.\ Rep.\ \textbf{27} (1972), 272--289.

\bibitem{DS2005} F.~Diamond and J.~Shurman, 
\emph{A First Course in Modular Forms},
Graduate Texts in Mathematics \textbf{228}, Springer-Verlag, New York, 2005. \mydoi{10.1007/978-0-387-27226-9}

\bibitem{E2013} W.~Ebeling, \emph{Lattices and Codes: A Course Partially Based on Lectures by Friedrich Hirzebruch},
third edition, Advanced Lectures in Mathematics. Springer Spektrum, Wiesbaden, 2013.  \mydoi{10.1007/978-3-658-00360-9}

\bibitem{H2005} T.~C.~Hales, \emph{A proof of the Kepler
    conjecture},
    Ann.\ of Math.\ (2) \textbf{162} (2005), no.\ 3, 1065--1185.
    \mydoi{10.4007/annals.2005.162.1065}

\bibitem{Hplus2017} T.~Hales, M.~Adams, G.~Bauer, T.~D.~Dang,
    J.~Harrison,
    L.~T.~Hoang, C.~Kaliszyk, V.~Magron, S.~McLaughlin, T.~T.~Nguyen,
    Q.~T.~Nguyen, T.~Nipkow, S.~Obua, J.~Pleso, J.~Rute,
    A.~Solovyev, T.~H.~A.~Ta, N.~T.~Tran, T.~D.~Trieu, J.~Urban, K.~Vu, and
    R.~Zumkeller, \emph{A~formal proof of the Kepler
    conjecture}, Forum Math.\ Pi \textbf{5} (2017), e2, 29 pp. \arXiv{1501.02155}
    \mydoi{10.1017/fmp.2017.1}
    
\bibitem{HMR2019} T.~Hartman, D.~Maz\'{a}\v{c}, and L.~Rastelli,
\emph{Sphere packing and quantum gravity},
J.\ High Energy Phys.\ \textbf{2019} (2019), no.\ 12, 048, 66 pp. 
    \arXiv{1905.01319} \\ 
    \mydoi{10.1007/jhep12(2019)048}
     
\bibitem{K2016} E.~Klarreich, \emph{Sphere packing solved in higher dimensions},
Quanta Magazine, March 30, 2016.
\url{https://www.quantamagazine.org/sphere-packing-solved-in-higher-dimensions-20160330/}

\bibitem{dLV2016} D.~de Laat and F.~Vallentin,
    \emph{A breakthrough in sphere packing: the search for magic
    functions}, Nieuw Arch.\ Wiskd.\ (5) \textbf{17} (2016), no.\ 3, 184--192. \\ 
    \arXiv{1607.02111} 
   
\bibitem{L-D2021} D.~Lowry-Duda, \emph{Visualizing modular forms}, 
in \emph{Arithmetic Geometry, Number Theory, and Computation} (J.~S.~Balakrishnan, N.~Elkies, B.~Hassett, B.~Poonen, A.~V.~Sutherland, and J.~Voight, eds.),
Simons Symposia, Springer, 2021. \\ 
\arXiv{2002.05234}
\mydoi{10.1007/978-3-030-80914-0\_19}

\bibitem{M2017} D.~Madore, \emph{Sections du diagramme de Vorono\"{i} du r\'eseau $E_8$}, 
David Madore's WebLog,
April 9, 2017. \\ 
 \url{http://www.madore.org/~david/weblog/d.2017-04-09.2433.html}

\bibitem{RV2019} D.~Radchenko and M.~Viazovska, \emph{Fourier interpolation
    on the real line},
    Publ.\ Math.\ Inst.\ Hautes \'Etudes Sci.\ \textbf{129} (2019), 51--81.
    \arXiv{1701.00265} \\ 
    \mydoi{10.1007/s10240-018-0101-z}
    
\bibitem{S1973} J.-P.~Serre, \emph{A Course in Arithmetic}, 
Graduate Texts in Mathematics \textbf{7}, Springer-Verlag, New York-Heidelberg, 1973. \mydoi{10.1007/978-1-4684-9884-4}
    
\bibitem{T1983} T.~M.~Thompson,
\emph{From Error-correcting Codes through Sphere Packings to Simple Groups},
Carus Mathematical Monographs \textbf{21}, Mathematical Association of America, Washington, DC, 1983.

\bibitem{T1892} A.~Thue, \emph{Om nogle geometrisk-taltheoretiske
    Theoremer}, 
    Forhandlingerne ved de Skandinaviske Naturforskeres \textbf{14} (1892),
    352--353.
    
\bibitem{V2017} M.~S.~Viazovska, \emph{The sphere packing problem in
    dimension $8$}, Ann.\ of Math.\ (2) \textbf{185} (2017), no.\ 3, 991--1015.
    \arXiv{1603.04246} \mydoi{10.4007/annals.2017.185.3.7}
    
\bibitem{Z2008} D.~Zagier, \emph{Elliptic modular forms and their
    applications}, in \emph{The 1--2--3 of Modular Forms} (K.~Ranestad, ed.), pp.~1--103, Universitext,
    Springer, Berlin, 2008.
    \mydoi{10.1007/978-3-540-74119-0\_1}
    
\end{thebibliography}
\end{document}